\documentclass[11pt,a4paper]{siamart171218}
\usepackage{amsmath,amstext,amsbsy,amssymb,graphicx,mathdots,todonotes}
\usepackage{tikz}
\usepackage{overpic} 
\usetikzlibrary{shapes,arrows}

\newcommand{\rr}[1]{\textcolor{black}{#1}}

\newcommand{\OurConstant}{4}

\newcommand{\bestlzero}{\smash{p_n^{\ell_0}\!}}
\newcommand{\bestlone}{\smash{p_n^{\ell_1}\!}}
\newcommand{\bestLzero}{\smash{p_n^{L_0}\!}}
\newcommand{\bestLinf}{\smash{p_n^{L_\infty}\!}}
\newcommand{\bestLone}{\smash{p_n^{L_1}\!}}
\newcommand{\pcheb}{\smash{p_n^{{\rm cheb}}}}
\newcommand{\betaConstant}{\beta}
\newcommand{\pLP}{\smash{p_n}}
\DeclareMathOperator*{\minimize}{min}

\DeclareMathOperator*{\subjectto}{subject\ to}

\usepackage{xcolor}

\definecolor{bluey}{rgb}{0.0,0.2,0.5}
\definecolor{bluey}{rgb}{0,0,0}

\definecolor{rev1}{rgb}{0.79,0.15,0.06}
\definecolor{rev1}{rgb}{0,0,0}
\newcommand{\RefOne}[1]{\color{rev1}#1}
\definecolor{rev2}{rgb}{0.8,0.8,0.01}
\definecolor{rev2}{rgb}{0,0,0}
\newcommand{\RefTwo}[1]{\color{rev2}#1}

\mathcode`@="8000
{\catcode`\@=\active\gdef@{\mkern1mu}}

\title{Error localization of best L1 polynomial approximants\thanks{Submitted to the editors \today.
\funding{The National Institute of Informatics in Tokyo partially funded an extended collaboration visit between the authors in December 2018, where the majority of this research took place. The first author is supported by the JSPS grants no.~17H01699 and 18H05837. The second author is supported by the National Science Foundation grant no.~1818757.}}}
\author{Yuji Nakatsukasa\thanks{
Mathematical Institute, University of Oxford, Oxford, OX2 6GG,
UK (\texttt{nakatsukasa@maths.ox.ac.uk}).
}
 \and Alex Townsend\thanks{
Department of Mathematics, Cornell University, Ithaca, NY 14853. (\texttt{townsend@cornell.edu}).}
}

\begin{document}
\maketitle

\begin{abstract}
An important observation in compressed sensing is that the $\ell_0$ minimizer of an underdetermined linear system is equal to the $\ell_1$ minimizer when there exists a sparse solution vector {\RefOne{and a certain restricted isometry property holds}}. Here, we develop a continuous analogue of this observation and show that the best $L_0$ and $L_1$ polynomial approximants of a polynomial that is corrupted on a set of small measure are nearly equal. We go on to demonstrate an error localization property of best $L_1$ polynomial approximants and use our observations to develop an improved algorithm for computing best $L_1$ polynomial approximants to continuous functions. 
\end{abstract}



\section{Introduction}
In compressed sensing the $\ell_0$ minimizer of an underdetermined linear system $Ax = b$ can be exactly recovered by the $\ell_1$ minimizer when the $\ell_0$ minimizer is sufficiently sparse and $A$ satisfies some regularity conditions~\cite{candes2006robust,donoho2006compressed,foucart2013mathematical}. Similarly, when an acquired signal is sparsely corrupted, one can exactly recover the original signal by minimizing the $\ell_1$ error, under suitable assumptions~\cite{candes2005decoding}.  In this paper, we investigate a continuous analogue of this phenomenon and show that the best $L_0$ and $L_1$ polynomial approximants of corrupted polynomials (see~\cref{def:CorruptedFunction}) are equal, under suitable assumptions (see~\cref{sec:L0L1}). We also make precise a related observation that the best $L_1$ error can be concentrated to intervals of small measure, showing that they can be advantageous compared to minimax approximants for certain applications (see~\cite{trefethen2011six}). 

Let $f:[-1,1]\rightarrow \mathbb{R}$ be a continuous function and $n\geq 0$ an integer. The best $L_1$ polynomial approximant, $\bestLone$, of degree $\leq n$ to $f$ exists, is unique~\cite[Thm.~14.3]{powell1981approximation}, and satisfies
\begin{equation} 
\|f - \bestLone\|_{1} = \min_{p\in\mathcal{P}_n} \| f - p\|_{1}, \qquad \| f - p\|_{1} = \int_{-1}^1 \left|f(x) - p(x)\right| dx,
\label{eq:L1problem} 
\end{equation} 
where $\mathcal{P}_n$ is the space of polynomials of degree $\leq n$. While the minimax approximant, $\bestLinf$, is the best approximant in the sense that $\|f - \bestLinf\|_{\infty} = \min_{p\in\mathcal{P}_n} \| f - p\|_{\infty}$, where $\|\cdot\|_{\infty}$ is the maximum norm, we know by the equioscillation theorem that the maximum deviation is attained $\geq n+2$ times~\cite[Thm.~7.2]{powell1981approximation}. On the other hand, it can frequently be observed that $|f(x) - \bestLone(x)| \ll \|f - \bestLinf\|_{\infty}$ for most, but not all, $x\in[-1,1]$ (see~\cref{fig:FirstExample} and~\cref{sec:theory}). To make this observation precise, we define the set\footnote{The constant of $1/2$ in the definition of $\Omega_n$ (see~\cref{eq:Omega}) is an arbitrary choice as any constant in $(0,1)$ would do, with very minor changes to the results that we derive.}
\begin{equation}
\Omega_n = \left\{x\in[-1,1] : |f(x) - \bestLone(x)| \geq \frac{1}{2}\|f - \bestLinf\|_{\infty}\right\}.
\label{eq:Omega}
\end{equation} 
For any $x\in [-1,1]\setminus \Omega_n$ we know that $\bestLone(x)$ is a better approximation to $f(x)$ than $\bestLinf(x)$. By the definition of $\bestLinf$, $\Omega_n$ is not the empty set, but we often observe that $|\Omega_n|\rightarrow 0$ as $n\rightarrow \infty$ (see~\cref{sec:theory}). For example, in~\cref{sec:theory} we prove that $|\Omega_n| = \mathcal{O}(n^{-2}\log n)$ for $f(x) = \sqrt{1-x^2}$ and $|\Omega_n| = \mathcal{O}(n^{-1})$ for $f(x) = |x|$. In such cases we say that the error $f-\bestLone$ is ``highly localized". This property of best $L_1$ approximation seems to be underappreciated and is related to observations from compressed sensing.
\begin{figure}
\centering 
\begin{overpic}[width=.95\textwidth]{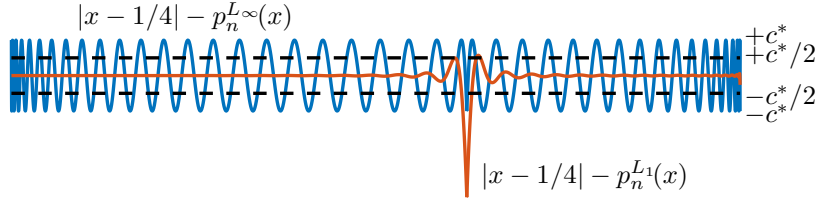}
\put(20,22) {$|x-1/4|-\bestLinf (x)$}
\put(63,5) {$|x-1/4|-\bestLone(x)$}
\put(91,20) {$+c^*$}
\put(91,18) {$+c^*/2$}
\put(91,13.5) {$-c^*/2$}
\put(91,11.5) {$-c^*$}
\end{overpic}
\caption{The errors $f(x)-\bestLinf(x)$ (blue line) and $f(x)-\bestLone(x)$ (red line) for $f(x) = |x-1/4|$ on $[-1,1]$ when $n = 80$. While $f(x)-\bestLinf(x)$ has a smaller absolute maximum on $[-1,1]$, we find that $|f(x)-\bestLone(x)|\leq c^*/2$ for most $x$ in $[-1,1]$, where $c^* = \|f - \bestLinf\|_{\infty}$. Similar illustrations can be found in~\cite[Chap.~16]{Trefethen_13_01} and~\cite{trefethen2011six}.}
\label{fig:FirstExample} 
\end{figure} 

The highly localized nature of $f-\bestLone$ means that best $L_1$ polynomial approximation is ideal for recovering functions that have been arbitrarily corrupted on a set of small measure. 
\begin{definition}
 For $0\leq s<1$, we say that a function $f:[-1,1]\rightarrow\mathbb{R}$ is a $s$-corrupted function 
 if $f$ can be written as 
\[
f(x) = g(x) + \omega(x),
\]
where $g:[-1,1]\rightarrow\mathbb{R}$ is a continuous function, $\omega(x)$ is a measurable function with $|{\rm supp}(\omega)| \leq s$, and $|{\rm supp}(\omega)|$ denotes the Lebesgue measure of the support of $\omega$ on $[-1,1]$. Note that the support of $\omega$, denoted by ${\rm supp}(\omega)$, is a closed subset of $[-1,1]$. 
\label{def:CorruptedFunction} 
\end{definition}
If $g=p_m$ is a polynomial of degree $\leq m$ in~\cref{def:CorruptedFunction}, then we say that $f$ is a corrupted polynomial. If, in addition, $s<\min(1,1/(\OurConstant n^2))$ for some integer $n\geq m$, then one finds that the best $L_1$ polynomial approximant of degree $\leq n$ to $f$ is unique and $\bestLone = p_m$ (see~\cref{cor:L1L0condition}). This means that best $L_1$ approximation exactly recovers a corrupted polynomial with arbitrary corruption, provided that the corruption has small enough support.

\begin{figure}
 \centering 
\begin{minipage}{.49\textwidth}
\begin{overpic}[width=\textwidth]{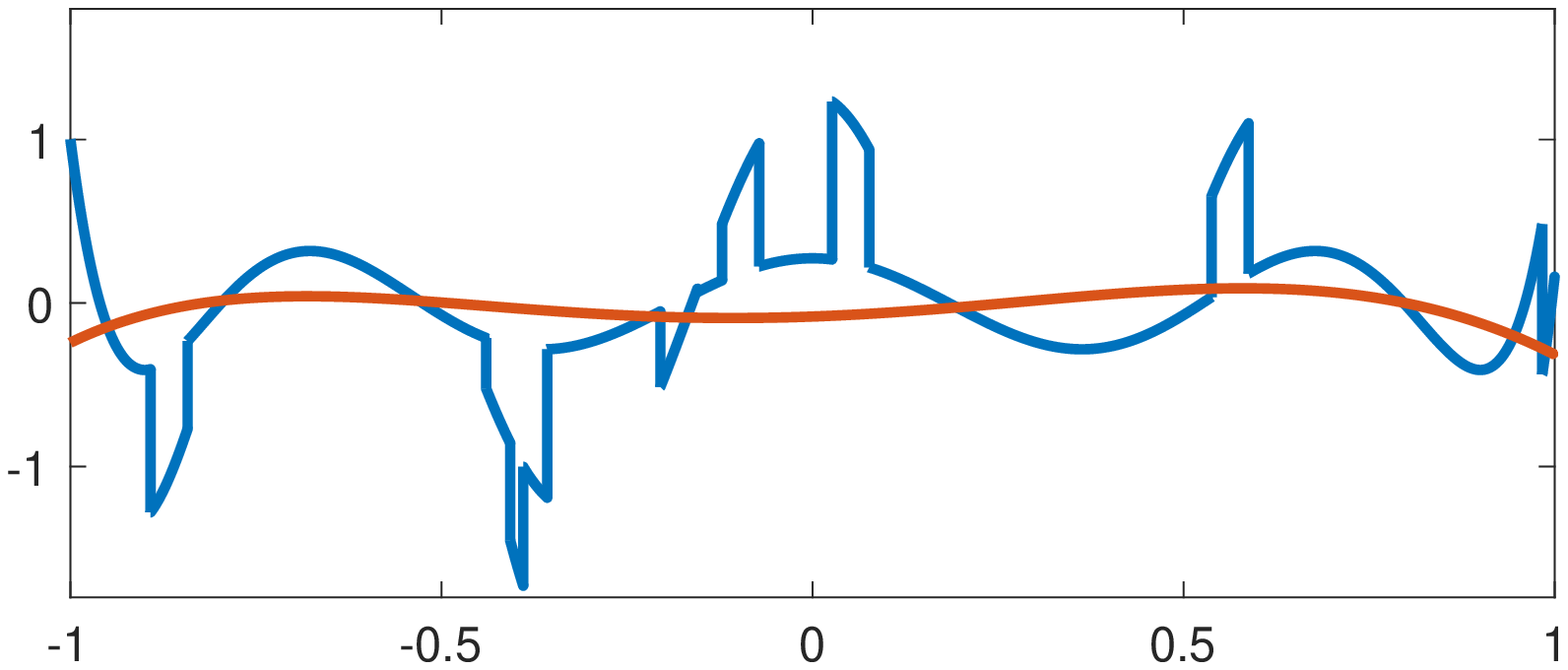} 
\put(45,30.5) {$n = 5$}
\put(3,30.5) {(a)}
\put(50,-3) {$x$}
\end{overpic}

\vspace{.2cm}

 \begin{overpic}[width=\textwidth]{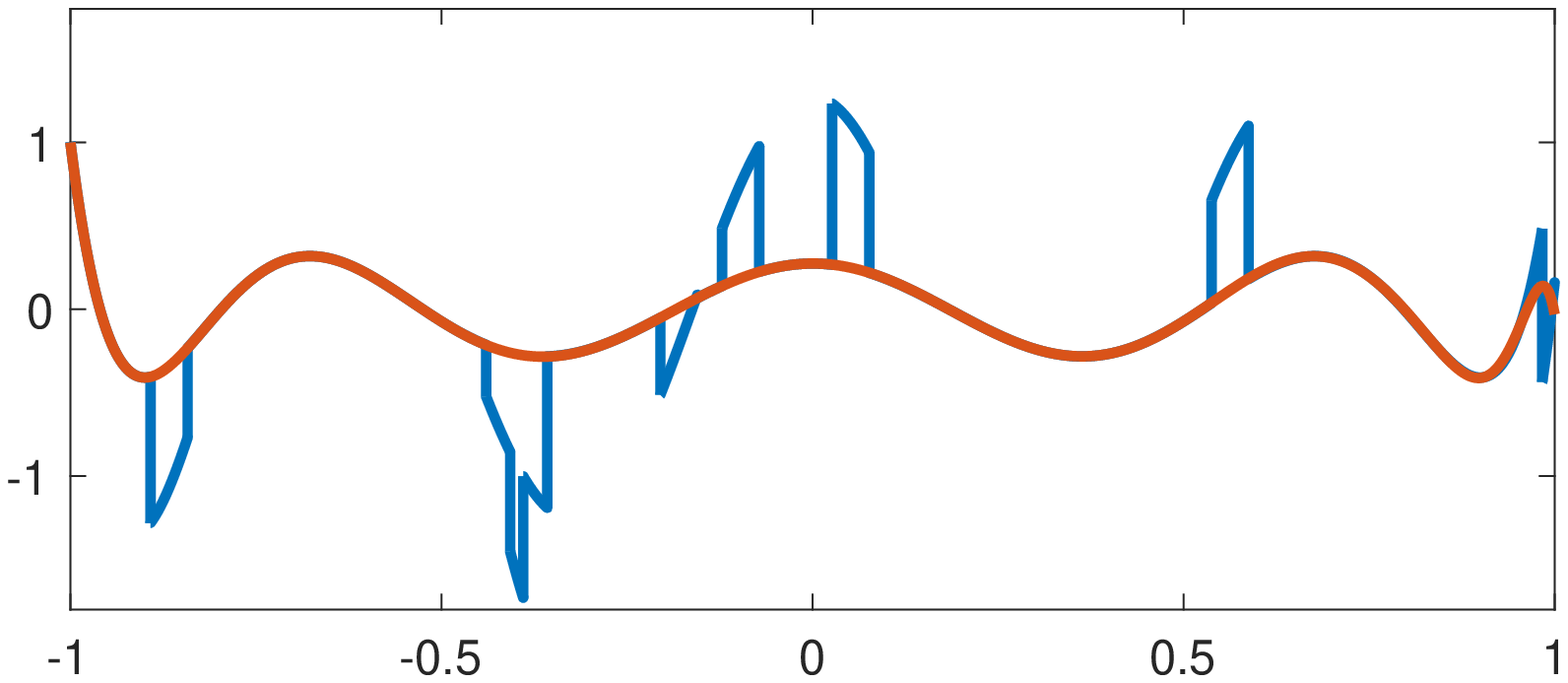}
\put(45,31) {$n = 20$}
\put(3,30.5) {(c)}
\put(50,-3) {$x$} 
\end{overpic}
  
\vspace{.2cm}

\end{minipage} 
\begin{minipage}{.49\textwidth}
 \begin{overpic}[width=\textwidth]{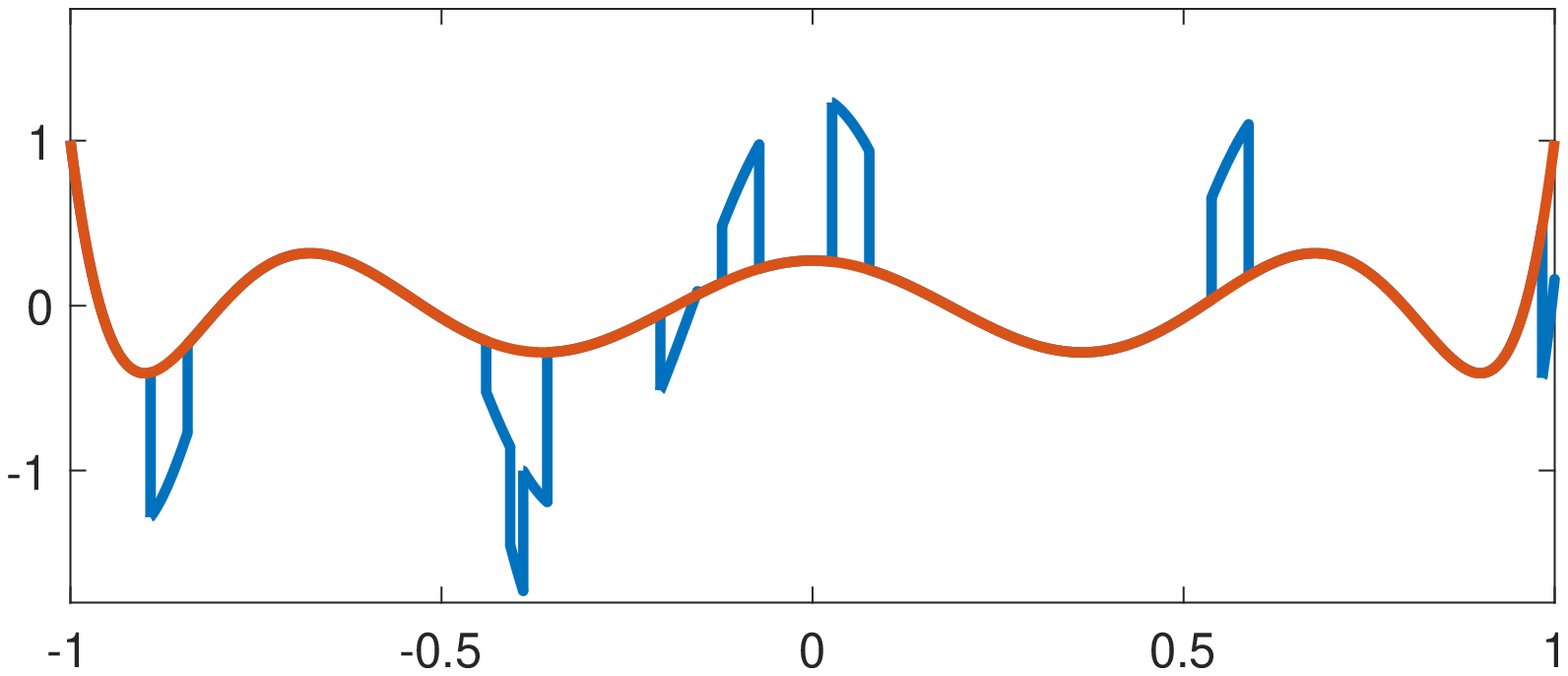} 
 \put(45,31) {$n = 10$}
 \put(3,30.5) {(b)}
 \put(50,-3) {$x$}
 \end{overpic}
 
\vspace{.2cm}

 \begin{overpic}[width=\textwidth]{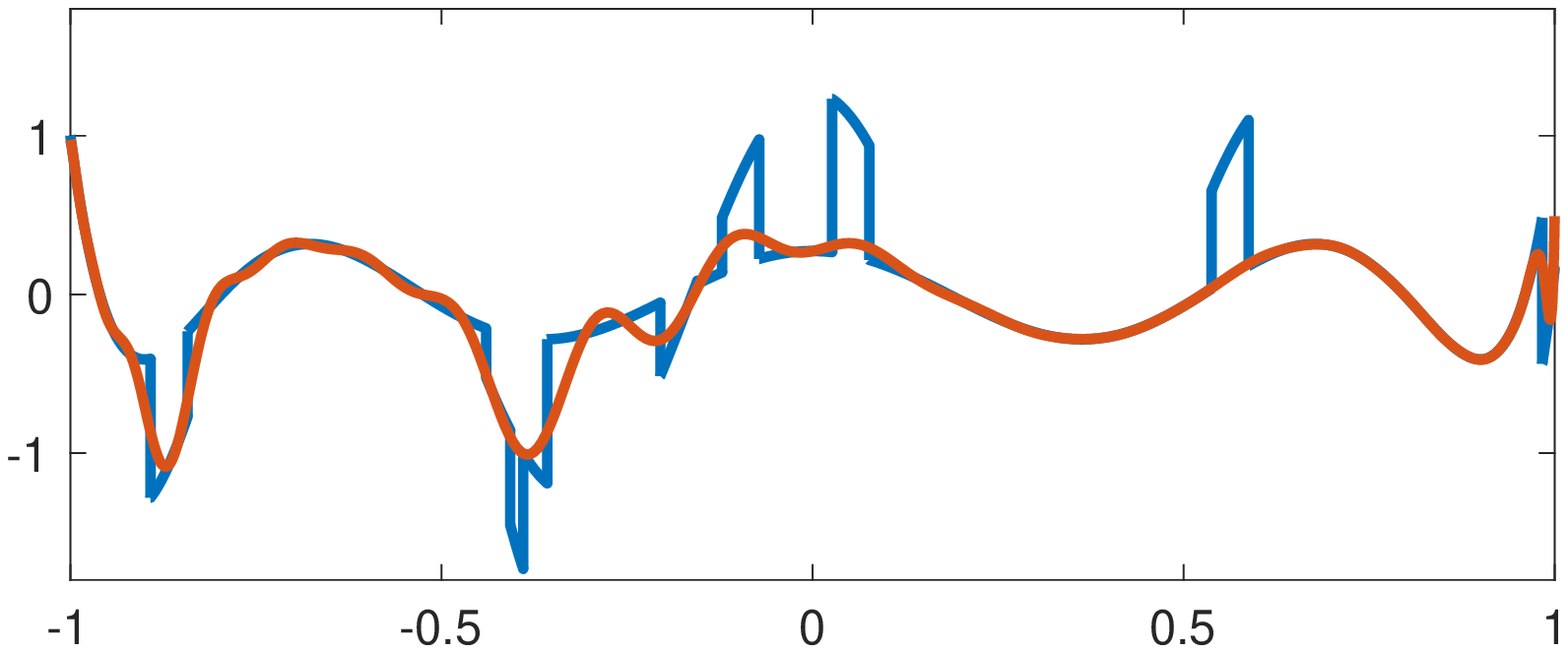} 
 \put(45,29.5) {$n = 40$}
 \put(3,30.5) {(d)}
 \put(50,-3) {$x$}
 \end{overpic}
 
\vspace{.2cm}

\end{minipage} 
 \caption{Best $L_1$ polynomial approximants of degree $n = 5$ (see (a)), $n=10$ (see (b)), $n=20$ (see (c)), and $n=40$ (see (d)) to a $s$-corrupted Legendre polynomial of degree $8$ with $s \approx 0.349$. For this example, we find the following: When $n<8$, $\bestLone$ does not recover the polynomial before it was corrupted (see (a)). When $8\leq n\leq 15$, $\bestLone$ perfectly recovers the polynomial before it was corrupted (see (b)). When $16\leq n\leq 26$, $\bestLone$ tries to fit corruptions near $\pm 1$ but not corruptions away from $\pm 1$ (see (c)). When $n>27$, $\bestLone$ tries to fit corruptions away from $\pm 1$ too (see (d)).}
 \label{fig:CorruptedFunction} 
\end{figure}
\Cref{fig:CorruptedFunction} illustrates the four regimes that one typically observes with best $L_1$ approximants of degree $\leq n$ of $f = p_m + \omega$: (a) If $n<m$, then $\bestLone\neq p_m$, but $\bestLone$ is a near-best approximant to $p_m$ (see~\cref{sec:L0L1smooth}), (b) If $n$ is small and $n\geq m$, then one gets exact recovery as $\bestLone = p_m$ (see~\cref{cor:L1L0condition}), (c) If $n$ is a little larger, then $\bestLone$ tries to fit corruptions near $\pm 1$ but not the corruptions away from $\pm 1$, and (d) When $n$ is large, $\bestLone$ tries to fit all the corruption, resulting in an overfit. 

We go on to derive an efficient algorithm for the recovery of $p_m$ from $f$ by showing that the continuous optimization problem in~\cref{eq:L1problem} for $\bestLone$ can be reduced to a linear programming problem, provided that a sampling condition is satisfied (see~\cref{thm:MainTheorem}). This observation results in a computationally efficient algorithm for the exact recovery of corrupted polynomials (see~\cref{sec:ell1ell0}).  

It is worth emphasizing that the Lebesgue measure of the support of the corruption must be extremely small. For example, our theory only guarantees that a corrupted polynomial of degree $100$ can be exactly recovered if it is corrupted on a set of measure $\leq 2.5\times 10^{-5}$. Nevertheless, in practice, we observe that exact recovery is usually still possible when the corruption occurs on sets that have a much larger measure. Moreover, the distribution of the corruption in $[-1,1]$ does matter.  In particular, larger regions of corruption are allowed away from $\pm1$ and we present an initial result in this direction (see~\cref{thm:HoleInTheMiddle}). For example, when $n = 100$ exact recovery is still guaranteed with any corruption interval of the form $[-s/2,s/2]$ with $s\leq 4\times 10^{-4}$.

The error localization properties of best $L_1$ approximants lead to an iterative algorithm for computing $\bestLone$ given a continuous function $f:[-1,1]\rightarrow \mathbb{R}$, based on a combination of linear programming and Newton's method (see~\cref{sec:alg}). This can be seen as an improvement on Watson's algorithm~\cite{Glashoff79,watson1981algorithm}. Our algorithm allows for the zero set of $f - \bestLone$ to have positive measure and heavily employs algorithmic advances over the last decade in polynomial rootfinding and adaptive Chebyshev interpolants~\cite{battles2004extension,pachon2010piecewise}. In particular, our implementation greatly benefits from the adaptive and robust algorithms for computing with functions in Chebfun.\footnote{Chebfun is an object-oriented software system written in MATLAB that provides an environment to compute with piecewise smooth functions~\cite{pachon2010piecewise}. It represents univariate functions defined on a finite interval by piecewise Chebyshev interpolants of adaptively selected degrees that are accurate to essentially machine precision~\cite{driscoll2014chebfun}.} It is able to accurately compute best $L_1$ approximants of degrees in the thousands (see~\cref{sec:alg}).

In addition to the $L_1$-norm (see~\cref{eq:L1problem}), we also {\RefTwo{define the following for}} continuous functions $f:[-1,1] \rightarrow \mathbb{R}$: 
\begin{equation}
  \label{eq:normdefs0}
 \begin{aligned}
  \| f \|_{\infty} &= \max_{x\in[-1,1]} |f(x)|, \qquad & \| f \|_{\ell_1} &= \sum_{j=0}^N w_j\left|f(x_j)\right|, \\
  \| f \|_{2} &= \RefTwo{\sqrt{\int_{-1}^1 f(x)^2dx}}, &   \| f \|_{\ell_0} &= \#\left\{ j : \left|f(x_j)\right| > 0, \quad 0\leq j\leq N \right\},    
 \end{aligned}  
\end{equation}
where $w_j\geq 0$ are weights so that $\sum_{j=0}^N w_j\left|f(x_j)\right|\rightarrow \int_{-1}^1|f(x)|dx$ as $N\rightarrow\infty$. {\RefTwo{Despite the notation, $\| \cdot \|_{\ell_0}$ is not a norm. For completeness, we also define $\|f\|_{0} = |{\rm supp}(f)|$ as the Lebesgue measure of the support of $f$.}} We always take $x_0,\ldots, x_N$ in the discrete norms $\| f \|_{\ell_1}$ and $\| f \|_{\ell_0}$ to be the roots of the degree $N+1$ Chebyshev polynomial of the second kind $U_{N+1}$~\cite[Tab.~18.3.1]{dlmf}. That is,  
\begin{equation}
x_j = \cos\!\left(\frac{(N+1-j)\pi}{N+2}\right), \qquad 0\leq j\leq N. 
\label{eq:ChebyshevGrid} 
\end{equation} 
Accordingly, we take $\smash{w_j=\pi \sqrt{1-x_j^2}/(N+2)}$ in~\cref{eq:normdefs0} so that the corresponding quadrature rule is related to the Gauss--Chebyshev rule.  The Chebyshev polynomials of the second kind and their roots in~\cref{eq:ChebyshevGrid} play a special role in best $L_1$ approximation~\cite[Ch.~14]{powell1981approximation}. In particular, when $N = n$, the polynomial interpolant of $f$ at the points in~\cref{eq:ChebyshevGrid}, i.e., 
\begin{equation} 
\pcheb(x) = \sum_{j=0}^n f(x_j)\ell_j(x), \qquad \ell_j(x) = \frac{\prod_{i=0,i\neq j}^n (x-x_i)}{\prod_{i=0,i\neq j}^n (x_j-x_i)},
\label{eq:ChebyshevInterpolant} 
\end{equation} 
is the best $L_1$ polynomial approximation of degree $\leq n$ to $f$ if $f-\pcheb$ has exactly $n+1$ distinct zeros in $[-1,1]$~\cite{brass1988remark,pinkus1989}.

For an integer $n\geq 0$, we denote by $\bestLinf$, $\smash{p_n^{L_2}}$, $\bestlone$, and $\bestlzero$ any best $L_\infty$, $L_2$, $\ell_1$, and $\ell_0$ polynomial of degree $\leq n$ to $f$, respectively.  These polynomials are solutions to the following optimization problems: 
\begin{equation} 
\begin{aligned} 
 \bestLinf &= \arg\min_{q\in\mathcal{P}_n} \|f - q\|_{\infty}, \qquad &\bestlone = \arg\min_{q\in\mathcal{P}_n} \|f - q\|_{\ell_1}, \cr
 p_n^{L_2} &= \arg\min_{q\in\mathcal{P}_n}\|f - q\|_{2}, \qquad &\bestlzero = \arg\min_{q\in\mathcal{P}_n} \|f - q\|_{\ell_0}.\cr
 \end{aligned}
 \label{eq:polynomialApproximants}
\end{equation}
{\RefTwo We also define $p_n^{L_0} = \arg\min_{q\in\mathcal{P}_n}\|f - q\|_{0}$, when best polynomial in this sense exists.}

The paper is structured as follows. In~\cref{sec:L0L1}, we show that the exact recovery of an arbitrarily corrupted polynomial is possible provided that the support of the corruption has small enough measure. This leads to an efficient algorithm to achieve recovery. In~\cref{sec:L0L1smooth}, we extend these ideas to the near-recovery of corrupted smooth functions. In~\cref{sec:theory}, we show that $|\Omega_n|$ is small precisely when $\|f - \bestLone\|_{1}\rightarrow 0$ faster than $\|f-\bestLinf\|_{\infty}\rightarrow 0$ as $n\rightarrow\infty$ and carefully consider two worked examples with error localization. Finally, in~\cref{sec:alg}, we present our iterative algorithm for computing best $L_1$ polynomial approximants of continuous functions.

\section{Exact recovery of corrupted polynomials}\label{sec:L0L1}
In this section we suppose that $f:[-1,1]\rightarrow \mathbb{R}$ is formed by an arbitrarily corrupted polynomial, i.e., $f = p_m + \omega$, where $p_m$ is a polynomial of degree $\leq m$ and $\omega$ is a function with small support. We investigate the question: When is it possible to exactly recover $p_m$ from knowledge of $f$? 

We show that for corrupted polynomials, we have $\bestLzero = \bestlzero = \bestlone= \bestLone = p_m$ provided that the support of $\omega$ is sufficiently small, $n\geq m$, and enough of the samples $x_0,\ldots, x_N$ in~\cref{eq:ChebyshevGrid} lie outside of the support of $\omega$. 
\begin{theorem} 
Let  $f=p_m+\omega$ be a $s$-corrupted polynomial of degree $\leq m$. Then, the following statements hold when $n\geq m$: 
\begin{enumerate} 
\item If $s<1$, then $\bestLzero  = p_m$.
\item If $(N-n)/2${\RefTwo, or fewer,} of the samples $x_0,\ldots, x_N$ are in ${\rm supp}(\omega)$, then $\bestlzero=p_m$.
\item If $k$ of $x_0,\ldots, x_N$ are in ${\rm supp}(\omega)$, $N+1>{\RefTwo 6}(n+1)k-1$, and $N\geq n$, then $\bestlone=p_m$. 
\item If $s <1/(n+1)^2$, then $\bestLone  = p_m$.
\end{enumerate} 
\label{thm:MainTheorem}
\end{theorem}
We prove the four statements in the theorem, in turn, in the next four subsections. 

\subsection{Exact recovery with best $\mathbf{L_0}$ approximation}\label{sec:ExactRecoveryL0} 
Intuitively, recovery of a corrupted function is ideal for best $L_0$ polynomial approximation as the Lebesgue measure of ${\rm supp}(f-\bestLzero)$ is minimized. The polynomial approximant $\bestLzero$ is so good at recovery that when $f = p_m+ \omega$ we have $\bestLzero=p_m$ provided that ${\rm supp}(\omega)$ is less than half the interval and $n\geq m$. 

To see this, note that $|{\rm supp}(f - p_m)| = {\rm supp}(\omega) = s<1$. Suppose there is a polynomial $q$ of degree $\leq n$ such that $|{\rm supp}(f - q)| \leq  |{\rm supp}(f - p_m)|$. Then, $|{\rm supp}(q - p_m)| \leq 2s < 2$ so $q$ and $p_m$ must coincide on a set of positive measure in $[-1,1]$. Since $q$ and $p_m$ are polynomials and $n\geq m$, we have that $q = p_m$. We conclude that $\bestLzero=p_m$ provided that ${\rm supp}(\omega)<1$ and $n\geq m$. This proves the first statement of~\cref{thm:MainTheorem}. 

\subsection{Exact recovery with best $\mathbf{\ell_0}$ approximation}\label{sec:L0ell0}
It can be algorithmically challenging to compute $\bestLzero$ and it is reasonable to attempt recovery from $\bestlzero$ instead, which involves a discrete optimization problem. The polynomial approximant $\bestlzero$ is also ideal at recovering polynomials under the mild assumption that enough of the samples $x_0,\ldots,x_N$ (see~\cref{eq:ChebyshevGrid}) lie outside of ${\rm supp}(\omega)$. 

To see this, suppose that $f = p_m + \omega$ and there is a polynomial $q$ of degree $\leq n$ such that 
\begin{equation} 
 \|f - q\|_{\ell_0} \leq \|f - p_m\|_{\ell_0} = k,  
\label{eq:ksparse} 
\end{equation} 
where $n\geq m$ and $k$ is the number of samples $x_0,\ldots, x_N$ in ${\rm supp}(\omega)$. If $k\leq (N-n)/2$, then $q-p_m\in\mathcal{P}_n$ is zero on at least $N+1-2k \geq n+1$ distinct points and hence $q=p_m$~\cite[p.~34]{powell1981approximation}. By definition of $\bestlzero$, we must have $\bestlzero = p_m$. This proves the second statement of~\cref{thm:MainTheorem}. 

\subsection{Exact recovery with best $\mathbf{\ell_1}$ approximation}\label{sec:ell1ell0}
The polynomial $\bestlzero$ can be computationally prohibitive to compute if $m$ is large.  Fortunately, by using the restricted isometry property (RIP) from compressed sensing, one finds that $\bestlzero=\bestlone$ when an oversampling condition is satisfied, along with some
regularity assumptions. This means that $\bestlone$, which can be computed efficiently, can often be used for exact recovery~\cite{mosek}. 

First, we know that $\|f-q_n\|_{\ell_0}=k$ for $q_n\in\mathcal{P}_n$ is equivalent to a vector $\underline{y}$ having precisely $k$ nonzero entries, where
\begin{equation}
  \label{eq:phimat}
 \underline{y} = \underbrace{\begin{bmatrix}U_0(x_0) \!\!&\!\! \cdots\!\! &\!\! U_n(x_0)\\[5pt] U_0(x_1) \!\!&\!\! \cdots\!\! &\!\! U_n(x_1)\\[5pt] \vdots\!\! &\!\! \ddots\!\! &\!\! \vdots \\[5pt] U_0(x_N)\!\! & \!\!\ldots\!\! & \!\!U_n(x_N) \end{bmatrix}}_{=\Phi}\!\!\!\begin{bmatrix}c_0\\[5pt]\vdots \\[5pt] c_n \end{bmatrix} - \begin{bmatrix}f(x_0)\\[5pt] f(x_1)\\[5pt]\vdots\\[5pt] f(x_N) \end{bmatrix}, \qquad q_n(x) = \sum_{i=0}^n c_iU_i(x)
\end{equation}
and $U_i(x)$ is the Chebyshev polynomial of the second kind of degree $i$~\cite[Tab.~18.3.1]{dlmf}. 
The problem of minimizing $\|\underline{y}\|_{\ell_0}$ over $\mathcal{P}_n$ in~\cref{eq:phimat} is solved by $\bestlzero$ and can be written as 
\begin{equation}
  \label{eq:ell0Phi}
\minimize_{\underline{c}\in\mathbb{R}^{n+1}}\|\Phi\underline{c}-\underline{f}\|_{\ell_0},\qquad \underline{f} = \begin{bmatrix}f(x_0)& \cdots & f(x_N) \end{bmatrix}^\top, 
\end{equation} 
which is equivalent to the following diagonally-scaled problem: 
\begin{equation}
  \label{eq:ell0PhiD}
\minimize_{\underline{c}\in\mathbb{R}^{n+1}}\|D\Phi \underline{c}-D\underline{f}\|_{\ell_0},\quad D = \sqrt{2/(N+2)}\,{\rm diag}\!\left(\sqrt{1-x_0^2},\ldots,\sqrt{1-x_N^2}\right). 
\end{equation} 

By a technique described in~\cite[p.~4204]{candes2005decoding}, if {\RefTwo $V\in\mathbb{R}^{(N+1)\times(N-n)}$} is a matrix whose columns form a basis for the left null space of $D\Phi$ so that $V^\top (D\Phi) = 0$, then~\cref{eq:ell0Phi} is also a constrained $\ell_0$ minimization problem:
\begin{equation}
  \label{eq:ymineq}
 \minimize_{\underline{z}\in\mathbb{R}^{N+1}} \|\underline{z}\|_{\ell_0},\qquad \mbox{subject to}\quad V^\top \underline{z} = -V^\top D\underline{f},
\end{equation}
where $\underline{z}=D\Phi\underline{c}-D\underline{f}$.  This problem is precisely the task of interest in the compressed sensing literature with a short-fat matrix $V^\top$ and an unknown sparse vector $\underline{z}$.

{\RefTwo The $\ell_0$ minimization problem~\eqref{eq:ymineq} is known to be NP-hard~\cite[Sec. 2.3]{foucart2013mathematical}. A practical remedy is to replace the $\ell_0$ norm with the $\ell_1$ norm. To understand when this gives the solution to the $\ell_0$ problem,}
an important concept in compressed sensing is the RIP. We say that a matrix $A\in\mathbb{C}^{m\times r}$ satisfies the RIP if there exists a constant $0<\delta_k<1$ such that 
\begin{equation} 
(1-\delta_k)\|\underline{x}\|_2^2 \leq \| A\underline{x}\|_2^2 \leq (1+\delta_k)\|\underline{x}\|_2^2,\qquad \|\underline{x}\|_2^2 = \sum_{i=1}^r |x_i|^2, 
\label{eq:RIP} 
\end{equation} 
for every vector $\underline{x}\in\mathbb{C}^r$ that has at most $k$ nonzero entries~\cite{candes2005decoding}.  It is known that if $V^\top$ satisfies the RIP with {\RefTwo $\delta_{k}<\frac{1}{3}$}, then the solution to~\cref{eq:ymineq} is exactly recovered (under the assumption that the $\ell_0$-minimizer $\leq k$ nonzero entries) by solving the $\ell_1$ minimization problem~\cite{cai2013sharp}
\begin{equation}\label{eq:LPy}
 \minimize_{\underline{z}\in\mathbb{R}^{N+1}} \|\underline{  z}\|_{\ell_1},\qquad \mbox{subject to}\quad V^\top \underline{ z} = -V^\top D\underline{f}.
\end{equation}
{\RefTwo Here,~\cref{eq:LPy} can be efficiently solved as a basis pursuit problem via the spectral projected-gradient $L_1$ (SPGL1) algorithm~\cite{BergFriedlander:2008}; especially, since there is a fast matrix-vector product for $V^\top$ based on the discrete sine transformation (see~\cref{eq:V}).} 

Note that unlike $\| f \|_{\ell_1}$ in~\cref{eq:normdefs0} for functions, the $\ell_1$ norm for vectors is simply the sum of the absolute values of the vector entries. The problem in~\cref{eq:LPy} is equivalent to 
\begin{equation}
  \label{eq:LP_solvethis}
 \minimize_{\underline{c}\in\mathbb{R}^{n+1}} \| D(\Phi \underline{c}-\underline{f})\|_{\ell_1}, 
\end{equation}
which in turn can be written as (recalling~\cref{eq:normdefs0}) the best $\ell_1$ approximation problem: 
\begin{equation}  \label{eq:discreteell1}
\minimize_{q\in\mathcal{P}_n} \|f - q\|_{\ell_1}. 
\end{equation}

We conclude that if the matrix $V^\top$ satisfies the RIP with {\RefTwo $\delta_{k}<\frac{1}{3}$} then we have $\bestlone=\bestlzero$, where
\[
\bestlone(x) = \sum_{j=0}^{n} c_j^* U_j(x)
\]
and the vector $\underline{c}^*$ is the solution to~\cref{eq:LP_solvethis}.

We are left with the task of studying when the matrix $V^\top$ in~\cref{eq:ymineq} satisfies the RIP with {\RefTwo $\delta_{k}<\frac{1}{3}$}. For the samples $x_0,\ldots,x_N$ that are given in~\cref{eq:ChebyshevGrid}, we have the discrete orthogonality condition $\smash{\sum_{\ell=0}^N U_{i}(x_\ell)U_j(x_\ell) (1-x_\ell^2)= 0}$ for $i\neq j$~\cite[Sec.~4.6.1]{mason2002chebyshev} so that we can write down an explicit basis for the left null space of $D\Phi$ in~\cref{eq:phimat}. That is,  
\begin{equation} 
 V = D \!\begin{bmatrix}U_{n+1}(x_0) & \ldots & U_{N}(x_0) \cr \vdots & \ddots & \vdots \cr U_{n+1}(x_N) & \ldots & U_{N}(x_N) \end{bmatrix} \in\mathbb{R}^{(N+1)\times(N-n)}. 
\label{eq:V}
\end{equation} 
It turns out that due to the choice of the diagonal matrix $D$ in~\cref{eq:ell0PhiD}, the matrix $V$ in~\cref{eq:V} is formed from a subset of columns of an orthogonal matrix. Furthermore, the size of $V^\top$ need not be extremely short-fat, as often required in compressed sensing.
It is therefore possible to show that $V^\top$ satisfies the RIP under a mild oversampling condition. 
\begin{proposition}\label{prop:Vtranspose}
If $N+2> 2(n+1)k$ for some integer $k\geq 1$, then $V^\top$ in~\cref{eq:V} satisfies the RIP with $\delta_k= (2(n+1)/(N+2))k$.
\end{proposition}
\begin{proof} 
Let $A$ be the $(N+1)\times (N+1)$ Chebyshev--Vandermonde matrix, i.e., $A_{ij} = U_j(x_i)$ for $0\leq i,j\leq N$, where $x_i$ is given in~\cref{eq:ChebyshevGrid}. Let $D$ be a diagonal matrix with $D_{i,i} = \sqrt{2/(N+2)}\sqrt{1-x_i^2}$ for $0\leq i\leq N$. By the discrete orthogonality properties of Chebyshev polynomials of the second kind~\cite[Sec.~4.6.1]{mason2002chebyshev}, $DA$ is an orthogonal matrix with 
\[
A^\top D= \begin{bmatrix} \Phi^\top D \\ V^\top \end{bmatrix}, 
\]
where $\Phi$ and $V$ are given in~\cref{eq:phimat} and~\cref{eq:V}, respectively.  Since $A^\top D$ has orthonormal columns, we find that 
\begin{equation}
  \label{eq:chebmatsing}
\left\|\begin{bmatrix} \Phi^\top D \\ V^\top \end{bmatrix}\!\underline{z}\right\|^2_2
=\|\Phi^\top D\underline{z}\|^2_2+\|V^\top\underline{z}\|^2_2=\|\underline{z}\|_2^2, \qquad \underline{z} \in \mathbb{C}^{N+1}.  
\end{equation}
Since $\sqrt{1-x^2}|U_i(x)| \leq 1$ for $x\in[-1,1]$~\cite[(18.14.7)]{dlmf}, each entry of $A^\top D$ has absolute value $\leq\sqrt{2/(N+2)}$ 
\rr{it follows by Cauchy--Schwarz that each 
entry of $\Phi^\top D\underline{z}$ is bounded by 
$\sqrt{\frac{2k}{N+2}}$ where $k$ is the number of nonzero entries in $\underline{z}$,}
so we have
\begin{equation} 
\|\Phi^\top D\underline{z}\|_2^2 \leq \frac{2(n+1)k}{N+2}\|\underline{z}\|_2^2.
\label{eq:FrobeniusInequality} 
\end{equation} 
 Therefore, from~\cref{eq:chebmatsing} and the trivial bound of $\|V^\top\underline{z}\|_2^2\leq \|\underline{z}\|_2^2$, we conclude that
\[
\left(1-\frac{2(n+1)k}{N+2}\right)\!\|\underline{z}\|_2^2\leq \|V^\top\underline{z}\|^2_2 \leq  \|\underline{z}\|_2^2
\]
for any vector $\underline{z}\in\mathbb{C}^{N+1}$ with at most $k$ nonzero entries. The statement immediately follows from the definition of the RIP (see~\cref{eq:RIP}). 
\end{proof}

{\RefTwo \Cref{prop:Vtranspose} tells us that $V^\top$ in~\cref{eq:V} satisfies the RIP with $\delta_{k} <1/3$ if $N+1>6(n+1)k-1$. Since $k$ is the number of samples $x_0,\ldots,x_N$ that lie in ${\rm supp}(\omega)$, it  means that $\bestlzero = \bestlone$ provided that the discrete problem is sufficiently oversampled. Since $k<(N+2)/(6(n+1))$ implies that $k\leq N-n$ when $k\geq 1$ and when $k = 0$ we need $N\geq n$, we conclude from~\cref{sec:L0ell0} that if $N+1>6(n+1)k-1$ and $N\geq n$, then $\bestlone = \bestlzero =p_m$ when $n\geq m$. This proves the third statement of~\cref{thm:MainTheorem}.}

The polynomial $\bestlone$ can be computed by solving the basis pursuit problem in~\cref{eq:LPy}. This means that \cref{prop:Vtranspose} gives us a practical and efficient algorithm for the exact recovery of corrupted polynomials with degrees in the thousands. Often it is the case that one does not know the degree of the corrupted polynomial or $k$. Since the oversampling condition $N+1>6(n+1)k-1$ penalizes taking unnecessarily large $n$, we recommend slowly increasing $n$, computing the error $f - \bestlone$, and stopping at the smallest $n$ for which ${\rm supp}(f - \bestlone)<2$.  


\subsection{Exact recovery with best $\mathbf{L_1}$ approximation}\label{sec:contcompressed_sens}
To begin to highlight the importance of error localization of best $L_1$ polynomial approximants, we now show that $\bestLone$ can also be used for exact recovery of corrupted polynomials when the corruption has sufficiently small support. One can achieve this by demonstrating that a polynomial of degree $\leq n$ is not too concentrated in any small subset of $[-1,1]$. 

\begin{lemma} \label{lem:L1L0}
Let $\Omega_s\subseteq [-1,1]$ be a set of Lebesgue measure $s\geq 0$. For any $n\geq 0$, we have
 \begin{equation} \label{eq:L1localizepoly}
  \int_{\Omega_s} \left|p(x)\right| dx \leq \frac{s(n+1)^2}{2}\int_{-1}^1 \left|p(x)\right| dx
 \end{equation}
 for any polynomial $p$ of degree $\leq n$. 
\end{lemma} 
\begin{proof} 
This statement is proved in~\cite[Sec.~4.2, Exercise 6]{benyamini20121}. 
\end{proof}

\Cref{lem:L1L0} tells us that polynomials of degree $\leq n$ cannot be too localized in a set of small measure. In particular, if $0\leq |\Omega_s|< 1/(n+1)^2$, then
\begin{equation} 
\int_{\Omega_s} |p(x)| dx \leq \int_{[-1,1]\setminus \Omega_s} |p(x)| dx, \qquad p\in\mathcal{P}_n,
\label{eq:L1energyhole} 
\end{equation} 
with equality if and only if $p$ is the zero polynomial. A consequence of~\cref{eq:L1energyhole} is that a corrupted polynomial can be exactly recovered by best $L_1$ polynomial approximation.  

\begin{corollary}\label{cor:L1L0condition}
Let $f = p_m + \omega$ be a $s$-corrupted polynomial of degree $\leq m$ on $[-1,1]$. Then, the best $L_1$ polynomial approximant of degree $\leq n$ to $f$ is $p_m$ if {\RefOne $n\geq m$} and $s<1/(n+1)^2$. 
\end{corollary}
\begin{proof} 
Let $\delta p\in\mathcal{P}_n$ and let $\Omega_s\subset [-1,1]$ be the support of $\omega$. Since $[-1,1] = \Omega_s \cup ([-1,1]\setminus \Omega_s)$, we have by the triangle inequality
\begin{equation} 
\begin{aligned} 
\|f - p_m - \delta p\|_{1} &= \int_{\Omega_s} \left|f(x) - p_m(x) - \delta p(x)\right| dx + \int_{[-1,1]\setminus\Omega_s} \left|\delta p(x)\right| dx\\
& \geq \int_{\Omega_s} \left|f(x) - p_m(x)\right| dx - \int_{\Omega_s} \left|\delta p(x)\right| dx + \int_{[-1,1]\setminus\Omega_s} \left|\delta p(x)\right| dx \\
&\geq  \|f - p_m\|_{1},
\end{aligned} 
\label{eq:BunchOfInequalities} 
\end{equation} 
where the last inequality follows from~\cref{eq:L1energyhole} as well as the fact that $f(x)-p_m(x)=0$ for $x\in [-1,1]\setminus\Omega_s$. An equality holds in~\cref{eq:BunchOfInequalities} if and only if $\delta p = 0$.  We conclude that $p_m$ is the unique best $L_1$ polynomial approximant to $f$ of degree $\leq n$.
\end{proof} 

This proves the fourth and final statement of~\cref{thm:MainTheorem} and explains regime (b) in~\cref{fig:CorruptedFunction}. It tells us that if a polynomial is corrupted on a subset of $[-1,1]$ that has small enough Lebesgue measure, then the best $L_1$ polynomial approximant exactly recovers the polynomial. \Cref{fig:corrupted} illustrates~\cref{cor:L1L0condition} for the corrupted polynomial $f = T_5 + \omega$, where $T_5$ is the degree $5$ Chebyshev polynomial of the first kind and ${\rm supp}(\omega) = [-.7,-.67]\cup[.9,.903]$. Using the fact that $\bestLone=\bestlone$, one can efficiently recover $T_5$ to within essentially machine precision. Numerically, we find that $\|\bestLone - T_5\|_{\infty} \approx 1.22\times 10^{-15}$. 

To highlight the importance of the $L_1$-norm for~\cref{cor:L1L0condition}, we consider the best polynomial approximants of degree $\leq 5$ to $f$ in the $L_2$- and $L_\infty$-norm (see~\cref{fig:corrupted} (right)). One finds that any corruption of arbitrarily small support prevents the best $L_2$ and $L_\infty$ polynomial approximants from recovering the uncorrupted polynomial.  
\begin{figure}
  \begin{minipage}{.49\linewidth}
\begin{overpic}[width=\textwidth]{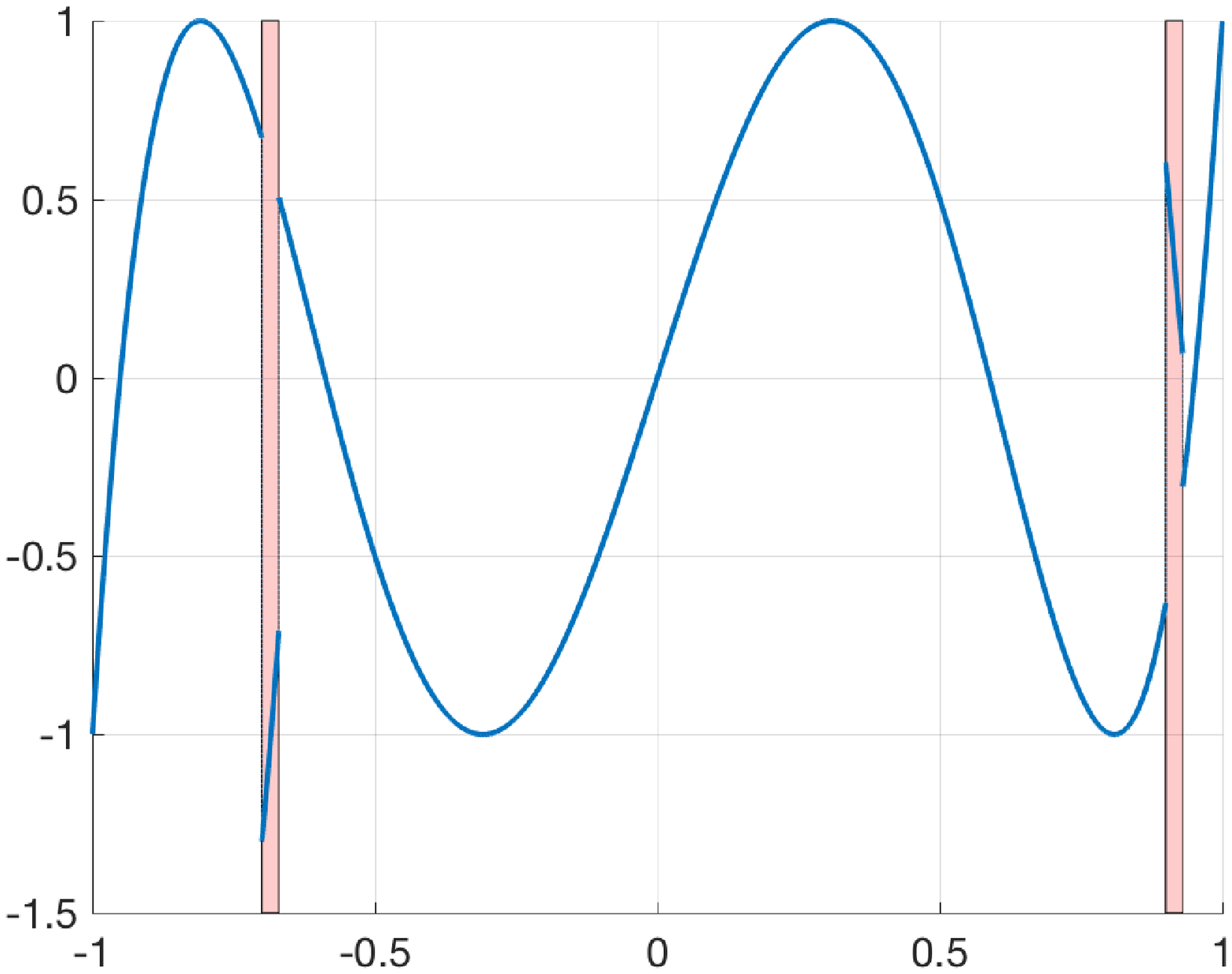}      
\put(50,0) {$x$}
\put(38,72) {$T_5(x) + \omega(x)$}
\put(20,25) {\rotatebox{90}{corruption}}
\put(82,25) {\rotatebox{90}{corruption}}
\end{overpic} 
  \end{minipage}
  \begin{minipage}{.49\linewidth}
  \begin{overpic}[width=\textwidth]{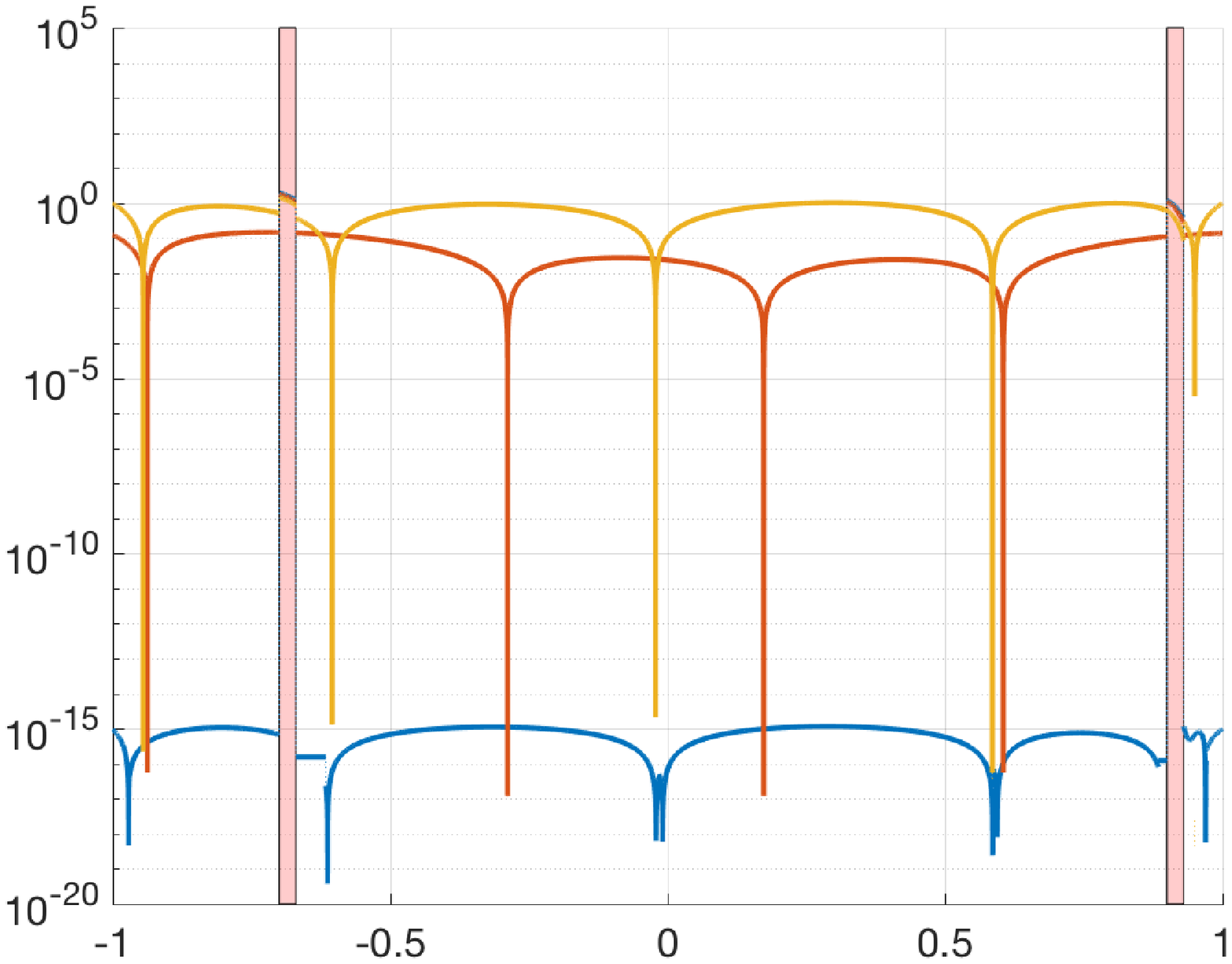}  
  \put(50,0) {$x$}    
  \put(20,25) {\rotatebox{90}{corruption}}
\put(82,25) {\rotatebox{90}{corruption}}
\end{overpic} 
  \end{minipage} 
  \caption{Left: Corrupted polynomial $f=T_5+\omega$, where $T_5$ is the degree $5$ Chebyshev polynomial of the first kind and ${\rm supp}(\omega) = [-.7,-.67]\cup[.9,.903]$ (shaded red). Right: The error $\smash{|f(x)-p^*_5(x)|}$, where $p^*_5$ is the best degree $\leq 5$ polynomial approximant to $f$ in the $L_1$-norm (blue line), $L_2$-norm (red line), and $L_\infty$-norm (yellow line). One can see that $\smash{|f(x)-p_5^{L_1}\!(x)|}$ is essentially machine precision for $x\not\in{\rm supp}(\omega)$ whereas $\smash{p_5^{L_2}\!}$ and $\smash{p_5^{L_\infty}\!}$ do not recover $T_5$.}
  \label{fig:corrupted}
\end{figure}

{\RefTwo The bound on $s$ of $s<1/(n+1)^2$ in~\cref{cor:L1L0condition} is probably not sharp. Though, we know that it cannot be increased above $\pi^2/(2(n+2)^2)$~\cite[Sec.~4.2]{benyamini20121}. This means that the algebraic scaling with respect to $n$ is definitive.} In~\cref{sec:middleCorruption}, we extend~\cref{cor:L1L0condition} by demonstrating that the location of the support of the corruption in $[-1,1]$ is important, and more is allowed provided that the corruption occurs away from $\pm 1$. 

For concreteness, we have assumed that the sample points are the Chebyshev points given in~\cref{eq:ChebyshevGrid}. This choice is recommended when the samples can be taken at arbitrary points in $[-1,1]$. However, in some cases, the sample points may be given a priori and cannot be chosen. Most of our results carry over to such cases with minor modifications and assumptions on the distribution of sample points. 

\section{Near-recovery of corrupted smooth functions}\label{sec:L0L1smooth}
When recovering a corrupted polynomial $f = p_m + \omega$, the degree of $p_m$ is usually unknown so we compute best $L_1$ polynomial approximants to $f$ of degree $\leq n$ for a slowly increasing sequence of $n$, stopping when ${\rm supp}(f - \bestLone)<2$. For the majority of this process $n<m$ and one may wonder what $\bestLone$ is achieving in this regime (see~\cref{fig:CorruptedFunction} (a)). Similarly, if $f$ is a corrupted smooth function $f = f_0 + \omega$, where $f_0$ is a continuous function (not necessarily a polynomial) on $[-1,1]$, then one cannot hope for exact recovery using best $L_1$ polynomial approximation. Instead, we find that $\bestLone$ delivers a near-recovery of $f_0$ in the sense that $\bestLone$ is a near-best $L_1$ approximation to $f_0$, provided that the support of the corruption is small and $f_0$ can be well-approximated by a degree $\leq n$ polynomial.  We first show that the best $L_1$ approximations for $f$ and $f_0$ are relatively close to each other. 

\begin{theorem}
Let $f = f_0 + \omega$ be a $s$-corrupted function on $[-1,1]$, where $f_0:[-1,1]\rightarrow\mathbb{R}$ is continuous, and $\bestLone$ be a best $L_1$ polynomial approximant of degree $\leq n$ to $f$. If $s<1/(n+1)^2$, then 
\[
\| \bestLone - p_n^* \|_{1} \leq \frac{4}{2-s(n+1)^2} \| f_0 - p_n^* \|_{1},
\]
where $p_n^*$ is the best $L_1$ approximant of degree $\leq n$ to $f_0$ on $[-1,1]$. 
\label{thm:RecoverySmoothFunctions} 
\end{theorem}
\begin{proof} 
Let $\delta p\in\mathcal{P}_n$ and $\Omega_s = {\rm supp}(\omega)$. Since $[-1,1] = \Omega_s \cup ([-1,1]\setminus \Omega_s)$, and by the triangle inequality, we have
\begin{align*}
  \|f-p_n^*-\delta p\|_{1}\! &= \! \int_{\Omega_s}\left|f(x)-p_n^*(x)-\delta p(x)\right| dx +\int_{[-1,1]\setminus\Omega_s}\!\!\!\!\!\! \left|f(x)-p_n^*(x)-\delta p(x)\right|dx\\
&\geq \!\int_{\Omega_s}\!\!\!\left(\left|f(x)-p_n^*(x)\right|-\left|\delta p(x)\right|\right) \!dx
+\!\int_{[-1,1]\setminus\Omega_s}\!\!\!\!\!\!\!\!\!\!\!\!\left(\left|\delta p(x)\right|-\left|f(x)-p_n^*(x)\right|\right)\!dx\\
&\geq \|f-p_n^*\|_{1}-2\|f_0-p_n^*\|_{1}+\int_{[-1,1]\setminus\Omega_s}\!\!\!\!\left|\delta p(x)\right|dx- \int_{\Omega_s}\left|\delta p(x)\right| dx,
\end{align*}
where the last inequality holds since $[-1,1] = \Omega_s \cup ([-1,1]\setminus \Omega_s)$ and $f(x)= f_0(x)$ for $x\not\in \Omega_s$.  From~\cref{eq:L1localizepoly}, we find that 
\[
\int_{\Omega_s}\left|\delta p(x)\right| dx\leq \frac{s(n+1)^2}{2} \|\delta p\|_{1}, \qquad \int_{[-1,1]\setminus\Omega_s}\left|\delta p(x)\right|dx \geq \left(1-\frac{s(n+1)^2}{2}\right)\!\|\delta p\|_{1}.
\]
Hence, for any $\delta p\in\mathcal{P}_n$ we have the inequality 
\[
\|f-p_n^*-\delta p\|_{1}\geq \|f-p_n^*\|_{1}-2\|f_0-p_n^*\|_{1}+ \left(1-\frac{s(n+1)^2}{2}\right)\!\|\delta p\|_{1}.
\]
Finally, by setting $\delta p =  \bestLone-p_n^*$ and noting that $\|f - \bestLone\|_{1}\leq \|f-p_n^*\|_{1}$ we conclude that
\[
-2\|f_0-p_n^*\|_{1}+ \left(1- \frac{s(n+1)^2}{2}\right)\|\bestLone-p_n^*\|_{1} \leq 0. 
\]
The result follows by rearranging this inequality. 
\end{proof}

\Cref{thm:RecoverySmoothFunctions} shows that best $L_1$ polynomial approximation is useful for near-recovery of a corrupted smooth function. More precisely, when $s<1/(n+1)^2$ we have 
\begin{equation} 
\| f_0-\bestLone \|_{1} \leq \left(1 + \frac{4}{2-s(n+1)^2}\right) \min_{q_n\in \mathcal{P}_n} \| f_0 - q_n \|_{1},
\label{eq:nearbest} 
\end{equation}
and we conclude that a best $L_1$ approximant of $f$ recovers $f_0$ as best it can, up to a factor that depends on $n$ and $s$. 

The inequality in~\cref{eq:nearbest} also partially explains regime (a) in~\cref{fig:CorruptedFunction}.  It provides theoretical justification that $p_5^{L_1}\!$ is a near-best polynomial approximant to $P_8$ in~\cref{fig:CorruptedFunction}. For the example in~\cref{fig:CorruptedFunction}, we observe this near-recovery phenomenon since
\[
\| P_8 - p_5^{L_1}\! \|_{1} \approx 0.450, \qquad  \min_{q_5\in \mathcal{P}_5} \| P_8 - q_5 \|_{1} \approx 0.414,
\]
where $p_5^{L_1}\!$ is the best $L_1$ approximant of degree $\leq 5$ to the corrupted function. 

Unlike corrupted polynomials (see~\cref{sec:L0L1}), $f_0$ cannot be exactly recovered by $\bestlone$. Nonetheless, we find that $\bestlone$ is often still a near-best approximant to $f_0$, i.e., $\bestlone\approx p_n^*$.  By interpreting $f_0-\bestlone$ as noise, we observe that $\ell_1$ minimization gives a stable signal recovery in the presence of noise, a phenomenon that is appreciated in the classical compressed sensing context~\cite{candes2006stable}. Making this observation precise in our setting is left as an open problem. Since by \cref{thm:RecoverySmoothFunctions} we also have $\bestLone \approx  p_n^*$, it follows that $\bestlone \approx \bestLone$ and $\bestlone$ is an excellent initial guess for Newton's method for computing $\bestLone$ (see~\cref{sec:newton}).

\subsection{Related studies}\label{sec:survey}
{\RefTwo{The contents of Sections~\ref{sec:L0L1} and \ref{sec:L0L1smooth} can be regarded as contributions in compressed sensing, and a number of related studies are available in the literature. 
For (exact and near-exact) recovery of corrupted functions with $\ell_1$ minimization, examples include the paper by Adcock, Brugiapaglia
and Webster~\cite{shin2016correcting}, and Shin and Xiu~\cite{adcock2019correcting}. Unlike this work, these papers consider recovering high-dimensional functions, describing probabilistic methods by taking random samples. 
Here we focus on univariate polynomials and reveal connections between $L_0,L_1,\ell_0$ and $\ell_1$ minimizers, and derive a deterministic recovery algorithm (under assumptions on the size of sampled corruption $k$) with $\ell_1$ minimization. Few of the results in this paper appear to be trivially generalizable to the higher-dimensional setting; this is left as an interesting open problem.

In the more classical setting of recovering a discrete signal (rather than a function) from a corrupted vector of observations, numerous contributions are available in the literature. See for example~\cite{candes2005error,candes2005decoding,laska2009exact,wright2010dense} and the references therein. Ideas in compressed sensing have also been applied for general high-dimensional function approximation~\cite{adcock2017compressed,cohen2015approximation}.}}

\section{Error localization of best $\mathbf{L_1}$ polynomial approximants}\label{sec:theory}
In~\cref{sec:L0L1,sec:L0L1smooth} we saw that $\bestLone$ can be used for recovering corrupted polynomials and smooth functions. This is fundamentally due to the error localization properties of best $L_1$ polynomial approximation. The error localization properties of $\bestLone$ are also important when approximating continuous functions $f:[-1,1]\rightarrow \mathbb{R}$ that one might not necessarily view as corrupted functions. We observe that continuous functions with singularities often have $|f(x) - \bestLone(x)| \ll \| f- \bestLinf \|_{\infty}$ for most $x\in[-1,1]$.  

To make this precise, recall the definition of $\Omega_n$ in~\cref{eq:Omega}. By definition of $\Omega_n$, we find that $\| f- \bestLone \|_{1} \geq \tfrac{|\Omega_n|}{2}\| f- \bestLinf \|_{\infty}$ and thus,
\begin{equation}
 0 < |\Omega_n|  \leq  2 \| f- \bestLone \|_{1}\Big/ \| f- \bestLinf \|_{\infty}.
\label{eq:EasyOmegaBound}
\end{equation} 
Therefore, the measure of $\Omega_n$ is bounded above by the disparity between the magnitude of $\| f- \bestLone \|_{1}$ and $\| f- \bestLinf \|_{\infty}$. If $\| f- \bestLone \|_{1}\rightarrow 0$ asymptotically faster than $\| f- \bestLinf \|_{\infty}\rightarrow 0$ as $n\rightarrow \infty$, then the error $f(x) - \bestLone(x)$ must be highly localized for sufficiently large $n$. An upper bound on $|\Omega_n|$ follows from an upper bound on $\| f- \bestLone \|_{1}$ and a lower bound on $\|f- \bestLinf \|_{\infty}$. 

\subsection{Error localization of best $\mathbf{L_1}$ approximants to $\mathbf{\sqrt{1-x^2}}$}\label{sec:ExampleSqrt}
Consider the function $f(x) = \sqrt{1-x^2}$, which is continuous on $[-1,1]$ with square root singularities at $\pm 1$. Here, we show that $|\Omega_n| = \mathcal{O}(n^{-2}\log n)$ proving that $\bestLone(x)$ is a better pointwise estimate to $f(x)$ than $\bestLinf(x)$ for all $x\in[-1,1]$ except for a set of measure $\mathcal{O}(n^{-2}\log n)$.

By~\cite[Lem.~4]{fiedler1990best}, we know that when $n$ is an even integer we have $\bestLone=\pcheb$ for $\sqrt{1-x^2}$, where $\pcheb$ is the degree $n$ Chebyshev interpolant of $\sqrt{1-x^2}$ (see~\cref{eq:ChebyshevInterpolant}). This allows us to derive an explicit expression for $\|f - \bestLone\|_{1}$ by using an explicit formula for $\|f - \pcheb\|_{1}$~\cite{brass1988remark}. By applying the formula in~\cite{brass1988remark} to $\sqrt{1-x^2}$, we find that  
\[
\|f - \bestLone\|_{1} = 2 \!\left|\sum_{\nu=0}^\infty (2\nu+1)^{-1} b_{(\nu+1)(n+2)-1}\right|\!, \quad b_j = \begin{cases} -\frac{8}{(j-1)(j+1)(j+3)\pi}, & j = \text{ even}, \\ 0, & j = \text{ odd}.\end{cases} 
\]
Here, the values of $b_j$ are derived as the expansion coefficients of $\sqrt{1-x^2}$ in a Chebyshev series of the second kind. That is, 
\[
\sqrt{1-x^2} = \sum_{j=0}^\infty b_j U_j(x), \qquad b_j = \frac{2}{\pi}\int_{-1}^1 (1-x^2) U_j(x) dx, \qquad j\geq 0. 
\]
Since $|b_j| \leq 16(j+1)^{-3}/\pi$ for $j>0$, we can bound $\|f - \bestLone\|_{1}$ by 
\[
\|f - \bestLone\|_{1} \leq \frac{32}{\pi(n+2)^3} \sum_{\nu=0}^\infty \frac{1}{(2\nu+1)(\nu+1)^3} \leq \frac{64}{\pi (n+1)^3},
\]
where the last inequality uses the crude bounds of $\sum_{\nu=0}^\infty (2\nu+1)^{-1}(\nu+1)^{-3} \leq 2$ and $n+2\geq n+1$.  

We now seek a lower bound on $\|f- \bestLinf \|_{\infty}$. Let $p_n^{\rm proj}(x) = \sum_{j=0}^n a_j T_j(x)$ be the Chebyshev expansion of the first kind for $\sqrt{1-x^2}$ that is truncated after $n+1$ terms. The values of $a_j$ are simple to calculate: $a_{2j-1}=0$ for all integers $j$, and
\[
a_0 = \frac{1}{\pi}\int_{-1}^1 T_0(x)dx = \frac{2}{\pi}, \qquad a_{2j} = \frac{2}{\pi}\int_{-1}^1 T_{2j}(x)dx = \frac{4}{(1-4j^2)\pi}, \quad j\geq 1.
\]
Assuming $n$ is an even integer, we find that
\[
p_n^{\rm proj}(1) = \sum_{j=0}^n a_j = \frac{2}{\pi} {\RefTwo +} \frac{4}{\pi}\sum_{j=1}^{n/2} \frac{1}{1-4j^2} = \frac{2}{\pi}\left(1-\frac{n}{n+1}\right) = \frac{2}{\pi (n+1)}.
\]
Thus, $\|f - p_n^{\rm proj}\|_{\infty} \geq 2/(\pi (n+1))$ for an even integer $n$. By~\cite[Cor.~4.1]{mason1983near}, we know that 
\[
\|f - p_n^{\rm proj}\|_{\infty} \leq (1 + \sigma_n) \|f - \bestLinf\|_{\infty}, \qquad \sigma_n = \frac{1}{\pi} \int_0^\pi \frac{\left|\sin(n + 1/2)\theta\right|}{\sin(\theta/2)}d\theta.
\]
We conclude from~\cref{eq:EasyOmegaBound} that for $f(x) = \sqrt{1-x^2}$ we have
\[
|\Omega_n| \leq \frac{128}{\pi (n+1)^3}  \frac{\pi (n+1)(1+\sigma_n)}{2} = \frac{64(1+\sigma_n)}{(n+1)^2} = \mathcal{O}(n^{-2}\log n), 
\]
where the final equality holds since it is known that $\sigma_n \sim 4\pi^{-2} \log n$~\cite[Eq.~20]{mason1983near}. 

\Cref{fig:sqrt1-x2} (left) shows the error $|f(x)-\bestLone(x)|$ for $x\in [0,1)$ demonstrating that it is localized near $x= \pm 1$. The measure of $|\Omega_n|$ is shown in \cref{fig:sqrt1-x2} (right) where it is numerically observed that $|\Omega_n| = \mathcal{O}(n^{-2})$. When $n =1000$, we find that $\smash{|f(x)-\bestLone(x)|<\tfrac{1}{2}\| f - \bestLinf \|_{\infty}}$ for all $x\in[-1,1]$ except for a set of measure $<10^{-5}$.
\begin{figure}
  \begin{minipage}[t]{0.495\hsize}
 \begin{overpic}[width=\textwidth]{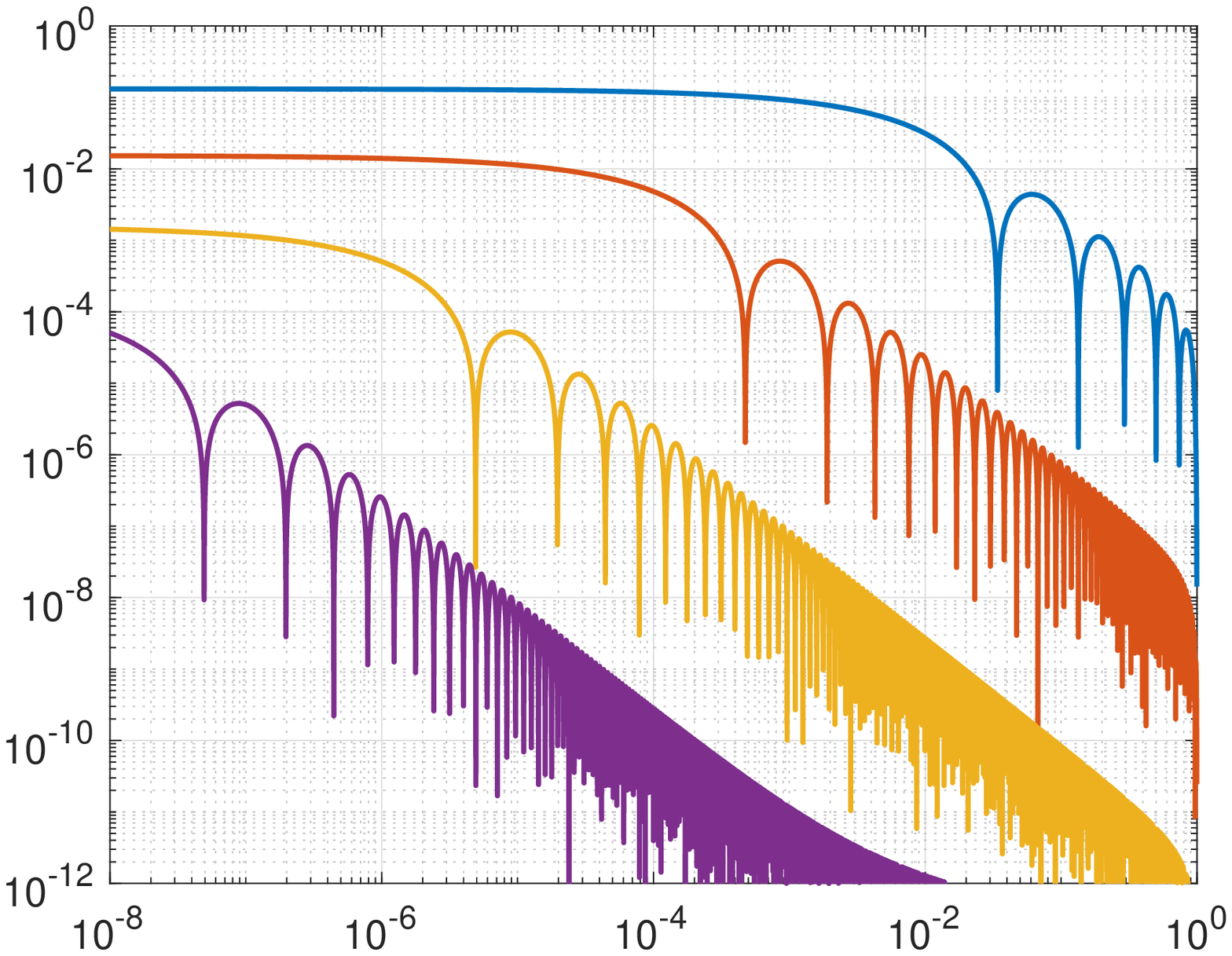} 
  \put(50,0) {$\epsilon$}
 \put(73,61) {\rotatebox{-28}{$n=10$}}
 \put(51,59) {\rotatebox{-28}{$n=100$}}
 \put(32,54) {\rotatebox{-28}{$n=1000$}}
 \put(15,48) {\rotatebox{-32}{$n=10000$}}
  \put(0,13) {\rotatebox{90}{$|f(1-\epsilon)-\bestLone(1-\epsilon)|$}}
  \put(34,72) {$f(x) = \sqrt{1-x^2}$}
 \end{overpic}
  \end{minipage}  
  \begin{minipage}[t]{0.495\hsize}
 \begin{overpic}[width=\textwidth]{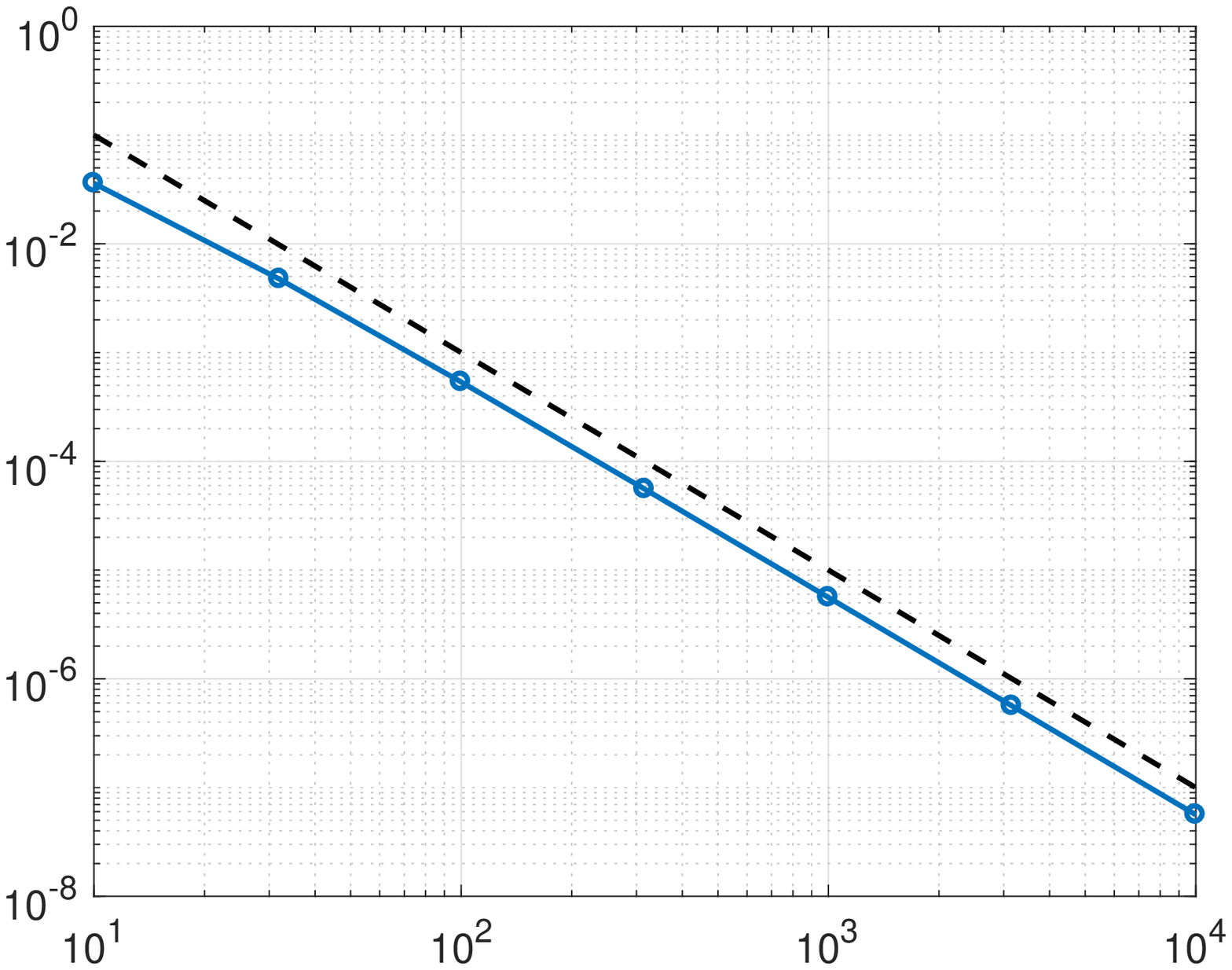} 
 \put(50,42) {\rotatebox{-28}{$\mathcal{O}(n^{-2})$}}
 \put(47,72) {$|\Omega_n|$}
 \put(50,0) {$n$}
 \end{overpic}
  \end{minipage}
\caption{Left: The error $|f(x)-\bestLone(x)|$ for $f(x)=\sqrt{1-x^2}$ with $n = 10$, $100$, $1000$, and $10000$, shown on the interval $[0,1)$. Right: It is observed that $|\Omega_n| = \mathcal{O}(n^{-2})$, showing that the error $|f(x)-\bestLone(x)|$ is highly localized. In particular, we find that $\smash{|f(x)-\bestLone(x)|<\tfrac{1}{2}\| f - \bestLinf \|_{\infty}}$ for all $x\in[-1,1]$ except for a set of measure $<10^{-5}$ near $x=\pm1$ when $n = 1000$.} 
 \label{fig:sqrt1-x2} 
\end{figure}

\subsection{Error localization of best $\mathbf{L_1}$ approximants to $\mathbf{|x|}$}\label{sec:ExampleAbs}
As a second example of error localization, consider $f(x) = |x|$ on $[-1,1]$, which is continuously differentiable except at $x = 0$. The error formula for $\| f-\bestLone \|_{1}$ with $f(x)=|x|$ is calculated in~\cite{brass1988remark} and simplifies to
\[
\| f-\bestLone \|_{1} \sim \frac{8}{\pi n^2}\left(\sum_{\nu =0}^\infty \frac{(-1)^\nu}{(2\nu+1)^3} \right) =  \frac{\pi^2}{4n^2}.
\] 
Moreover, it is known that $\| f - \bestLinf \|_{\infty}\! \sim \frac{\betaConstant}{2n}$ for some $0.28016<\beta<0.28018$~\cite{varga1985bernstein}. We conclude from~\cref{eq:EasyOmegaBound} that $|\Omega_n| \lesssim \frac{\pi^2}{\betaConstant n}$ as $n\rightarrow \infty$.  \Cref{fig:abs} (left) shows the error $|f(x)-\bestLone(x)|$ for $x\in [0,1)$ demonstrating that it is highly localized and~\cref{fig:abs} numerically confirms that $|\Omega_n| = \mathcal{O}(n^{-1})$. 
\begin{figure}
  \begin{minipage}[t]{0.495\hsize}
 \begin{overpic}[width=\textwidth]{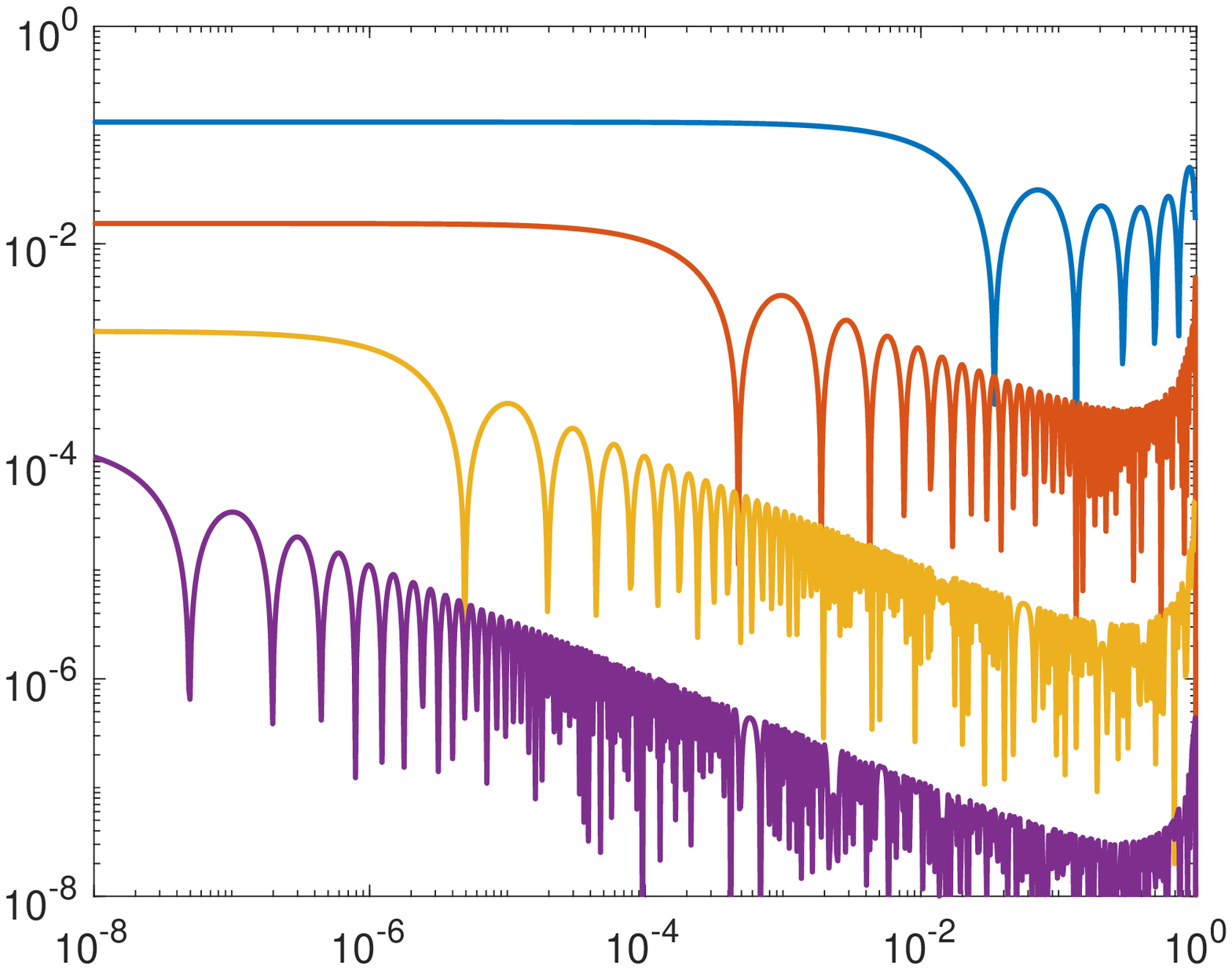} 
  \put(50,0) {$\epsilon$}
  \put(73,61) {\rotatebox{-12}{$n=10$}}
 \put(65,50) {\rotatebox{-20}{$n=100$}}
 \put(54,39) {\rotatebox{-21}{$n=1000$}}
 \put(40,29) {\rotatebox{-21}{$n=10000$}}
  \put(0,13) {\rotatebox{90}{$|f(1-\epsilon)-\bestLone(1-\epsilon)|$}}
   \put(40,72) {$f(x) = |x|$}
 \end{overpic}
  \end{minipage}   
  \begin{minipage}[t]{0.495\hsize}
 \begin{overpic}[width=\textwidth]{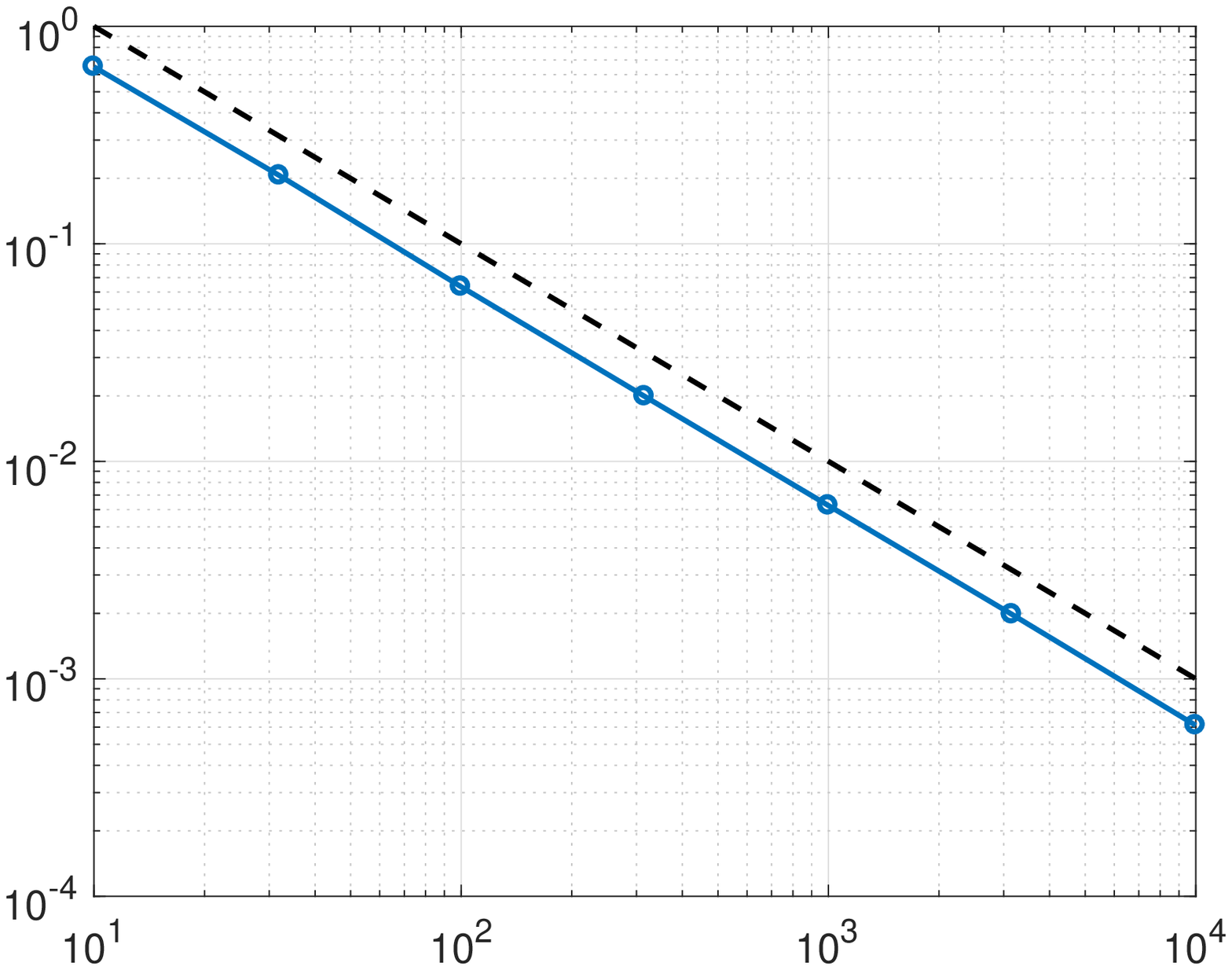} 
  \put(50,50) {\rotatebox{-28}{$\mathcal{O}(n^{-1})$}}
   \put(47,72) {$|\Omega_n|$}
 \put(50,0) {$n$}
 \end{overpic}
  \end{minipage}
 \caption{Left: The error $|f(x)-\bestLone(x)|$ for $f(x)=|x|$ with $n = 10$, $100$, $1000$, and $10000$ shown on the interval $[0,1)$. Right: It is observed that $|\Omega_n| = \mathcal{O}(n^{-1})$. In this example, we find that the error $|f(x)-\bestLone(x)|$ is highly localized near $x=0$ and $x=\pm1$.}
 \label{fig:abs} 
\end{figure}

\section{A globally convergent algorithm for computing best $\mathbf{L_1}$ polynomial approximants}\label{sec:alg}
We now turn to the algorithmic aspects of computing $\bestLone$. We integrate our findings on exact recovery of corrupted polynomials and error localization into Watson's algorithm based on Newton's method~\cite{watson1981algorithm}. An algorithm to compute best $L_1$ approximants with degrees in the thousands is developed based on recent advances in approximation theory such as stable polynomial interpolation, fast domain subdivision, and robust rootfinding implemented in Chebfun~\cite{driscoll2014chebfun}. \Cref{fig:diagram} gives an overview of our algorithm.  
\begin{figure}
\centering
{\footnotesize{
\tikzstyle{decision} = [diamond, draw, fill=white!20,
    text width=6em, text badly centered, node distance=2.1cm, inner sep=0pt]
\tikzstyle{block} = [rectangle, draw, fill=white!20,
    text width=7em, text centered, node distance=2.4cm, rounded corners, minimum height=4em]
\tikzstyle{line} = [draw, thick, color=black, -latex']

\begin{tikzpicture}[scale=2, node distance = 0cm, auto]
    \node [block,text width=5em, node distance=1cm] (prob) {Compute $\pcheb$ of $f$ (see~\cref{eq:ChebyshevInterpolant})};
    \node [decision,right = .35cm of prob,text width=4.5em ,node distance=1cm] (n+1?) {Does $f-\pcheb$ have $n+1$ roots?};
    \node [block, below = .35cm of n+1?,text width=6em, node distance=1cm] (trivial) {$\bestLone=\pcheb$};
    \node [block, right = .45cm of n+1?,text width=4.5em, node distance=1cm] (lp) {Solve \cref{eq:LPy} then LP with refined mesh};
    \node [block,below = .95cm of lp,text width=4.5em,node distance=1cm] (solnlp) {$\bestLone=\bestlone$};
    \node [block, right = .35cm of lp,text width=4em, node distance=1cm] (newton) {Newton's method 
};
    \node [decision,right = .35cm of newton,text width=5.5em ,node distance=1cm] (conv) {Converged?};
    \node [block,below = .7cm of conv,text width=4.5em,node distance=1cm] (soln) {$\bestLone=\pLP$};
\draw[->,thick] (prob) -- (n+1?);
\node[,above right = .4cm and -.5cm of newton](subdivide){no, iterate};
\node[,below right = .8cm and -.4cm of conv](yes){yes};
\draw[->,thick] (n+1?) -- (lp);
\node[,below right = .6cm and -.6cm of n+1?](triv){yes};
\node[,above right = -.5cm and .5cm of n+1?](nontriv){no};
\draw[->,thick] (n+1?) -- (trivial);
\draw[->,thick] (lp) -- (solnlp);
\node[,below right = .1cm and -.95cm of lp](corrupt){\begin{tabular}{l}corrupted\\ polynomial\end{tabular}};
\draw[->,thick] (lp) -- (newton);
\draw[->,thick] (newton) -- (conv);
\draw[->,thick] (conv) -- (soln);
\draw[|-,-|,->, dashed, thick,-latex] (conv.north)  |-+(0,.25em)-|  (newton.north);
\end{tikzpicture}
}}
\caption{Flowchart for our algorithm to compute the best $L_1$ polynomial approximant of degree $\leq n$ to a continuous function $f$ on $[-1,1]$.}
\label{fig:diagram}
\end{figure}
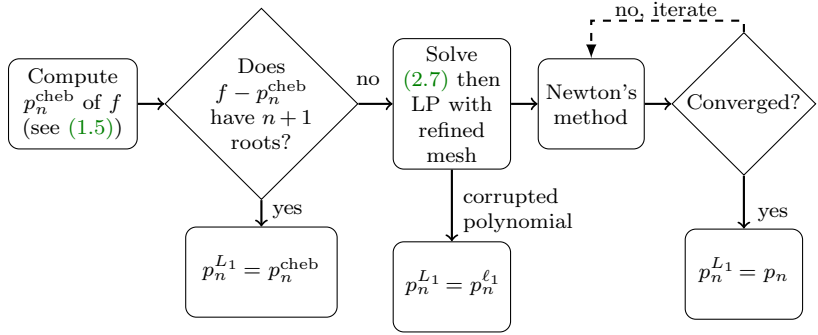

\subsection{Initial attempt: The Chebyshev interpolant}\label{sec:trivial}
The polynomial interpolant $\pcheb$ in~\cref{eq:ChebyshevInterpolant} with $N = n$ can be computed in $\mathcal{O}(n\log n)$ operations~\cite{gentleman1972implementing} and the roots of $f - \pcheb$ on $[-1,1]$ can be computed efficiently when $f$ is a smooth function~\cite{boyd2002computing}. Since $\bestLone = \pcheb$ when $f - \pcheb$ has exactly $n+1$ roots in $[-1,1]$~\cite{pinkus1989}, we recommend that~\cref{eq:ChebyshevInterpolant} is always computed to see if $\bestLone = \pcheb$. When it is, $\bestLone$ is efficient to compute and from practical experience it is relatively common for $\bestLone = \pcheb$ (for example, see~\cite[Lem.~4]{fiedler1990best}). This can happen also when $f$ is a corrupted polynomial.

\subsection{Test for corrupted polynomials and initial guess: Compute $\mathbf{\ell_1}$ minimizer}\label{sec:LP}
When $f - \pcheb$ has $>n+1$ zeros in $[-1,1]$, computing $\bestLone$ is more involved and, in general, requires an iterative procedure. 
In this case, we first solve the discrete $\ell_1$ problem in \cref{eq:discreteell1} to obtain $\bestlone$. This has two purposes: (i) If $f$ is a corrupted polynomial $f=p_m + \omega$ (see \cref{sec:L0L1}), then $\bestlone=p_m=\bestLone$, and (ii) If $f$ is not a corrupted polynomial, then $\bestlone\approx \bestLone$~\cite[Thm.~3.9]{rivlin2003introduction}, which is then used as the initial guess for Newton's method (see~\cref{sec:newton}). 

Specifically, we solve the LP in~\cref{eq:ell1LP} with a large number of samples $N+1$, taking $x_0,\ldots,x_N$ and $\smash{w_j=\pi \sqrt{1-x_j^2}/(N+2)}$ as in~\cref{eq:ChebyshevGrid}.  In our implementation we select $N+1=\max(1000+50n,5000)$. {\RefTwo{(This is an engineering choice that assumes the corruption $k$ is small.)}}
Recall from \cref{thm:MainTheorem} that we want $N+1>6(n+1)k-1$.) 
The maximum value 5000 is set to keep the LP size $2(N+1)+n+1$ manageable.

Once $\bestlone$ is computed, we check whether $f$ is a corrupted polynomial. This can be done by testing if $f(x_j)=\bestlone(x_j)$ holds at most of the sample points to within working precision. If not, then we improve the estimate $\bestlone\approx \bestLone$ by refining the LP mesh, and then proceed to Newton's method. 

\subsubsection{Refinement: Reducing the discretization error}\label{sec:lpm2}
Underlying the minimization problem~\cref{eq:LPy} is an approximate integration of a  non-differentiable function. Specifically, 
\begin{equation}\label{eq:ell1}
\min_{p_n\in\mathcal{P}_n} \! \sum_{i=0}^{N} w_i\left|f(y_i)-p_n(y_i)\right|, \quad \int_{-1}^1 \!\left|f(x)-p_n(x)\right| dx \approx \sum_{i=0}^N w_i\left|f(y_i)-p_n(y_i)\right|.
\end{equation}
Since $|f(x)-p_n(x)|$ is expected to be continuous, but non-differentiable at $\geq n+2$ points, one expects the integration error in~\cref{eq:ell1} to be large and there is little benefit from using a high-order quadrature rules. Indeed using $N+1$ sample points, we find that the LP solution has accuracy $\|\bestlone-\bestLone\|_{1}=\mathcal{O}(N^{-1})$, whether a high-order method (e.g.~Clenshaw-Curtis) or a low-order method (such as the midpoint rule) is used. In more detail, the quadrature error in \cref{eq:ell1} is $\mathcal{O}(N^{-2})$, so the objective function value $\|f-\bestlone\|_1$ is within $\mathcal{O}(N^{-2})$ of optimal: $\|f-\bestlone\|_1=\|f-\bestLone\|_1+\mathcal{O}(N^{-2})$. However, this only implies $\|\bestLone-\bestlone\|_1=\mathcal{O}(N^{-1})$, which is a common phenomenon in optimization: at a global (or local) minimum, an $\epsilon$-perturbation in the solution results in $O(\epsilon^2)$ perturbation in the objective value. This low accuracy of $\bestlone$ can cause convergence issues for Newton's method, when it is used as an initial guess. 

To improve the discretization error in~\cref{eq:ell1}, we follow a three-step procedure: (1) We use the initial LP solution with $N$ points to obtain an $\mathcal{O}(N^{-1})$ approximation to $\bestLone$, which we denote by $\tilde p_n$. (2) The roots $\{r_i\}_{i=1}^K$ of $f - \tilde p_n$ in $[-1,1]$ are computed, which we expect to be $\mathcal{O}(N^{-1})$ approximations to the roots of $f-\bestLone$. 
Finally, (3) we solve another LP to obtain $\bestlone$, which is a better approximant to $\bestLone$ than $\tilde p_n$, with a discretization scheme that forms a finer mesh near the roots: We take $\approx N/2$ points on $\cup_{i=1}^K[r_i-\delta,  r_i+\delta]$, where $\delta = 4/N$, taking equispaced points on each subinterval. We then take $\approx N/2$ more points on $[-1,1]$, outside the subintervals, again uniformly, i.e., the grid is much coarser (see \cref{fig:LPconv} (right)). We take the weights $w_j$ according to the midpoint rule. 
{\RefTwo{We thus take a mesh $O(1/N^2)$ rather than $O(1/N)$ near the roots, while still having a $O(1/N)$ mesh elsewhere.}}
This refinement of the quadrature rule is observed to improve the accuracy to $\|\bestlone-\bestLone\|_1=\mathcal{O}(N^{-2})${\RefTwo{, as the quadrature error 
at the roots have been improved from $\mathcal{O}(N^{-2})$ to $\mathcal{O}(N^{-4})$. 
We then solve~\cref{eq:LPy} by a standard technique of casting it as linear programming (LP)~\cite{candes2005decoding}, namely 
\begin{equation}\label{eq:ell1LP}
  \begin{split}
&\minimize_{u_0,\ldots,u_N,v_0,\ldots,v_N, c_0,\ldots,c_n}\quad  \sum_{i=0}^N w_i(u_i+v_i),\\    
&\subjectto \quad u_i\geq 0,\quad  v_i\geq 0, \quad -v_i\leq f(y_i)-\sum_{j=0}^n c_jU_j(y_i) \leq u_i,\quad 0\leq i\leq N,
  \end{split}
\end{equation}
Note that we do not use SPGL1 or the Chebyshev points from~\cref{eq:phimat} in the refinement stage. This is because SPGL1 requires the computation of the null space $V^\top$, which can be more expensive. 
Due to the sparsity structure of LP, we find that the MOSEK optimization toolbox~\cite{mosek} (using its MATLAB interface) offers an efficient solver. 
}}

In~\cref{fig:LPconv} (left) we show the error $\|\bestLone-\bestlone\|_{1}$ \rr{with the LP solution} for $10^2\leq N\leq 10^4$, with and without the refinement. Note that the number of decision variables in LP~\cref{eq:ell1LP} is $2(N+1)+n+1$, with $4(N+1)$ inequality constraints. 

\begin{figure}
\centering
  \begin{minipage}[t]{0.495\hsize}
\vspace{-0mm}
 \begin{overpic}[width=.9\textwidth]{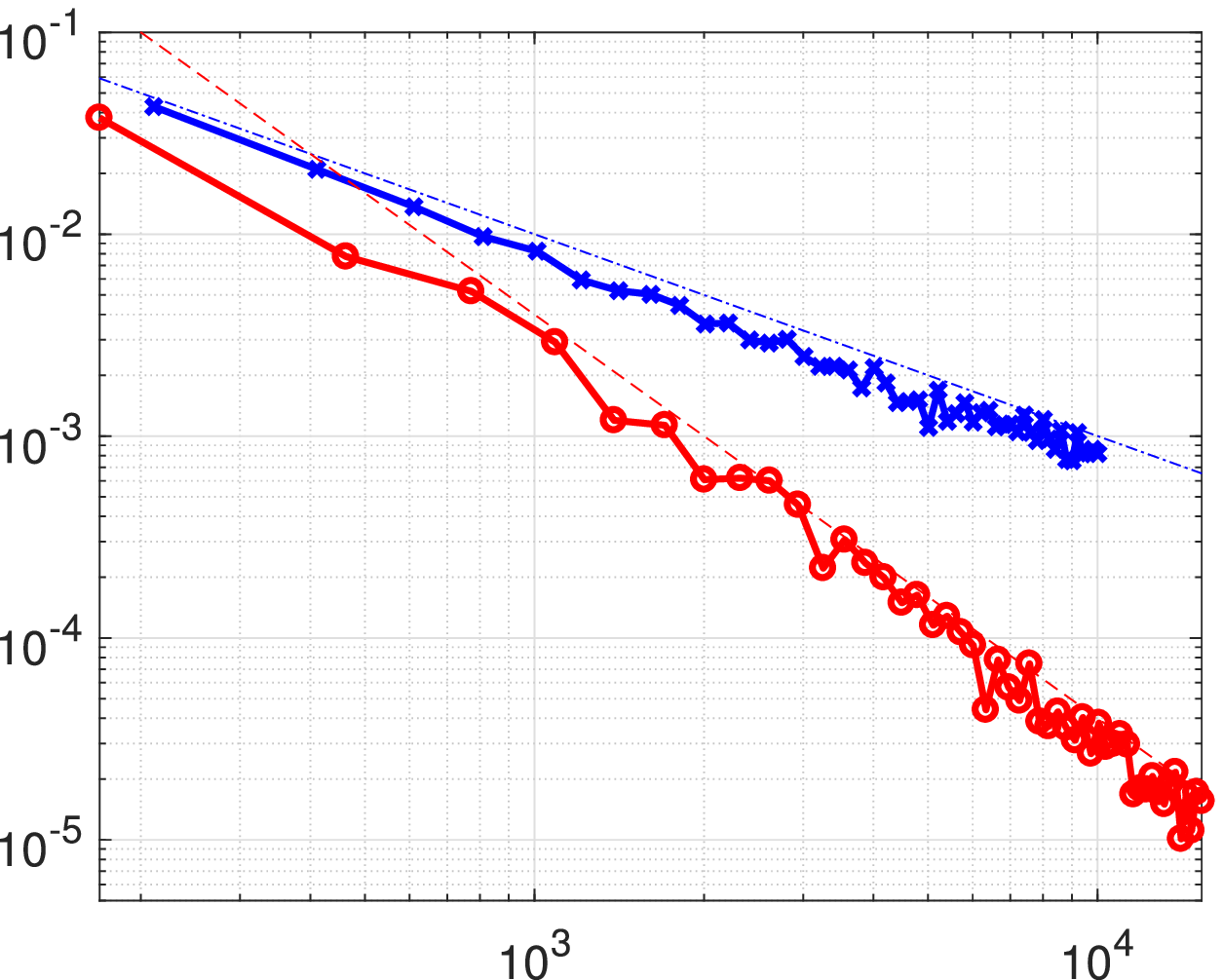}
 \put(48,-1) {LP size $\approx 2N$}
 \put(-7,28) {\rotatebox{90}{$\|\bestLone - \bestlone\|_{1}$}}
 \put(59,45) {\rotatebox{-31}{$\mathcal{O}(N^{-2})$}}
  \put(64,56) {\rotatebox{-20}{$\mathcal{O}(N^{-1})$}}
 \put(70,24) {\rotatebox{-35}{refined LP}}
  \put(84,38) {\rotatebox{-20}{LP}}
 \end{overpic}
  \end{minipage}
  \begin{minipage}[t]{0.495\hsize}%
\vspace{1.5mm}
  \begin{overpic}[width=.9\textwidth]{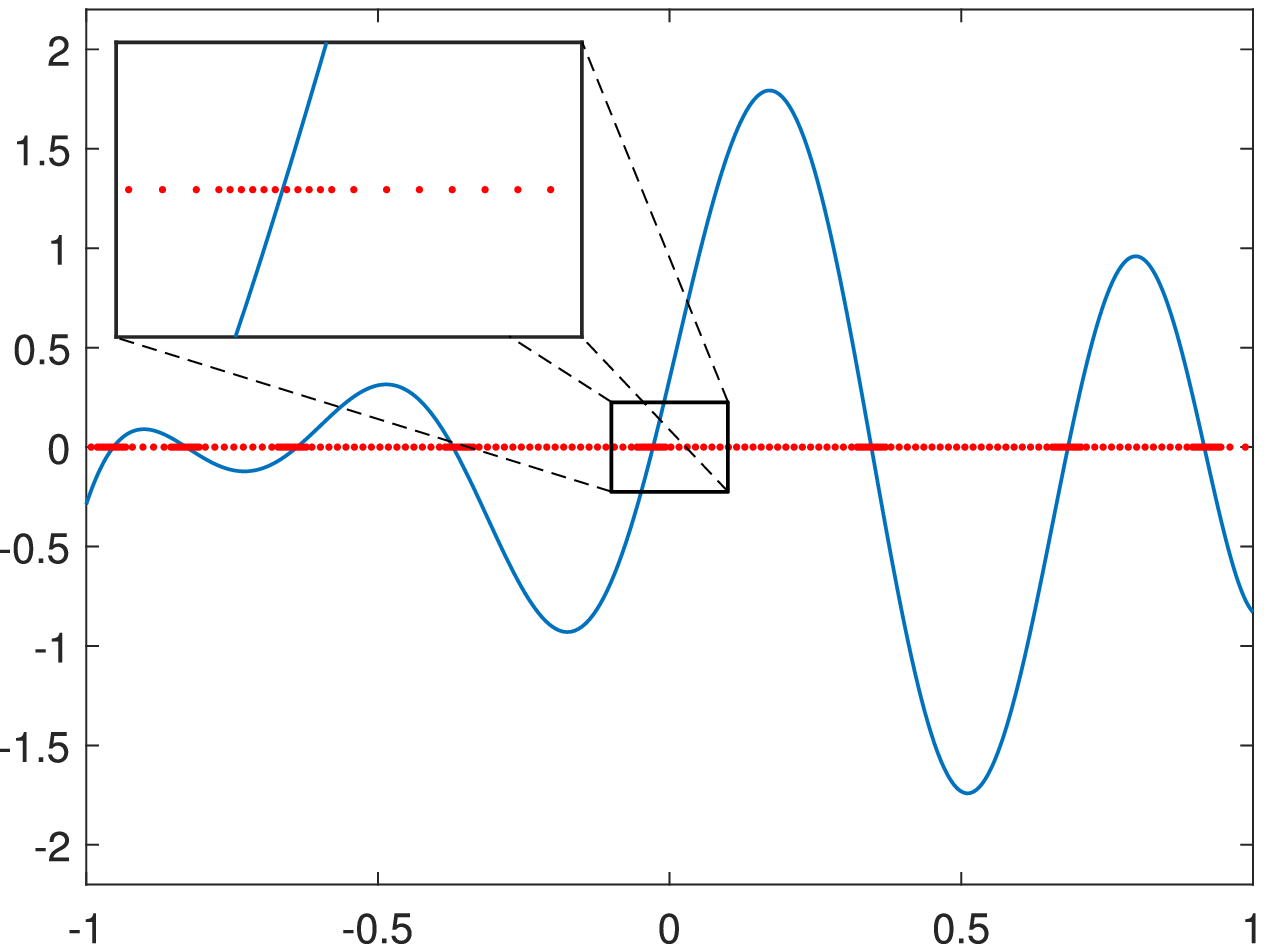} 
 \put(65,66) {\rotatebox{0}{$f(x)-\tilde p_n(x)$}}
 \put(51,-4) {$x$}
 \end{overpic}
  \end{minipage}
  \caption{Left: The error $\|\bestLone - \bestlone\|_{1}$ for $f(x)=\exp(x)\sin(10x)$ and $n=10$ with and without refinement compared against the number of LP variables, which is roughly $2N$. Here, $m$ is the number of sample points used to discretize the continuous $L_1$ optimization problem. Right: Sample points (red dots) used in the refined LP with $n=5$. The mesh is much finer near the roots of $f-\tilde p_n$, so that the discretization error is significantly reduced. Here, $\tilde p_n$ is the solution of the first (unrefined) LP.}
  \label{fig:LPconv}
\end{figure}

\subsection{Iterative procedure: Newton's method}\label{sec:newton}
To improve the initial guess obtained in~\cref{sec:LP} we employ Newton's method based on the ideas in Watson's algorithm~\cite[Sec.~4]{watson1981algorithm}, which is a globally convergent (under mild assumptions) iterative method for computing $\bestLone$ when the set $S = \{x\in[-1,1] : f(x) = \bestLone(x)\}$ has zero Lebesgue measure. \rr{We assume this below; otherwise $f$ was a corrupted polynomial, which would be detected by~\cref{eq:LPy} if the corruption is small. 
}

When the set $S$ has zero Lebesgue measure, an alternative characterization of $\bestLone$ is~\cite[Thm.~14.1]{powell1981approximation}
\begin{equation}\label{eq:iffcond2}
\int_{-1}^1 s(x)q(x)dx =0 , \quad s(x) = {\rm sign}(f(x)-\bestLone(x)) = \begin{cases} 1, & f(x)-\bestLone(x)\geq 0, \cr 0, &f(x)-\bestLone(x) = 0,\cr -1,& f(x)-\bestLone(x)<0, \end{cases} 
\end{equation}
for all $q\in\mathcal{P}_n$. We propose to apply Newton's method to~\cref{eq:iffcond2}. By using the Chebyshev polynomials of the second kind as a basis for $\mathcal{P}_n$, we define a vector-valued operator $L: \mathbb{R}^{n+1} \mapsto \mathbb{R}^{n+1}$ given by 
\begin{equation}
  \label{eq:mujdef}
L\!\left[(c_0,\ldots,c_n)^\top\right] = \begin{pmatrix}\mu_0,\ldots,\mu_n\end{pmatrix}^\top, \qquad  \mu_j = \int_{-1}^1 {\rm sign}\!\left(f(x)-\sum_{j=0}^n c_j U_j(x)\right)\!  U_i(x) dx.   
\end{equation}
We note that $L[(c_0^*,\ldots,c_n^*)^\top] = \underline{0}$ if and only if $\bestLone = \sum_{j=0}^n c_j^* U_j$ from~\cref{eq:iffcond2}, and we propose to use Newton's method on $L$ to find it.

Newton's method tells us to perform the following iteration: 
\begin{equation} 
\underline{c}^{(k+1)} = \underline{c}^{(k)} - J_k^{-1} L[\underline{c}^{(k)}], \quad (J_k)_{i,j} = \frac{\partial}{\partial c_j^{(k)}}\int_{-1}^1 {\rm sign}\!\left(f(x)-\sum_{t=0}^n c_t^{(k)} U_t(x)\right)\!  U_i(x) dx.
\label{eq:iteration} 
\end{equation} 
Moreover, it can be shown that $J_k$ can be expressed as~\cite{watson1981algorithm}
\begin{equation} 
J_k =2V_k^\top \mbox{diag}\!\left(\frac{1}{e_k'(r_1)},\ldots,\frac{1}{e_k'(r_K)}\right)\!V_k,\qquad e_k(x) = f(x) - \sum_{t=0}^n c_t^{(k)} U_t(x),
\label{eq:Jk} 
\end{equation} 
where $r_1,\ldots,r_K$ are the roots of $e(x)$ and $V_k$ is the Chebyshev--Vandermonde matrix at $r_1,\ldots,r_K$, i.e., $(V_k)_{i,j} = U_j(r_i)$.

At the $k$th Newton iteration, we must calculate the roots of $e_k(x) = f(x) - \smash{\sum_{t=0}^n c_t^{(k)} U_t(x)}$, evaluate 
\rr{$\mu_j$ for $0\leq j\leq n$ and} 
$e_k'(x)$ at $r_1,\ldots,r_K$, form $J_k$ using~\cref{eq:Jk}, and then solve an $(n+1)\times (n+1)$ dense linear system where the righthand side is $L[\underline{c}]$. All these operations can be performed conveniently and robustly in Chebfun to an accuracy of essentially machine precision~\cite{driscoll2014chebfun}. 
{\RefOne{The dominant computation in each Newton's step lies either in the evaluation of $\mu_j$ in~\eqref{eq:mujdef}, which costs
$\mathcal{O}(nm^2)$ where $m$ is the Chebfun degree of $f$, or the linear system $\mathcal{O}(n^3)$, for a total of $\mathcal{O}(n^2(m+n))$ complexity. Typically Newton converges within a handful of iterations. 
}}

As Watson notes~\cite{watson1981algorithm}, a small modification for the formula for $J_k$ in~\cref{eq:Jk} is required when $e_k'(r_j)=0$ for some $r_j$, e.g., set $J=I$, or when $V$ is rank-deficient, e.g., set $J:=J+\delta I$ for some small $\delta>0$. Under mild restrictions, this modified Newton's method generically converges to $\bestLone$ at a quadratic rate~\cite{watson1981algorithm}. 

\subsection{Stopping criterion: Near-best condition}\label{sec:stopping} 
It is important to have a stopping criterion to determine when Newton's method in~\cref{eq:iteration} should be terminated. The simplest criterion could be to stop computing iterates as soon as $\|\underline{c}^{(k+1)} - \underline{c}^{(k)}\|_2 < \epsilon \|\underline{c}^{(k)}\|_2$, where $\epsilon>0$ is a small parameter. However, we prefer to stop Newton's method as soon as $\max_{0\leq i\leq n} \left|(L[\underline{c}^{(k)}])_i\right| < \epsilon\|f\|_{1}$ because it leads to a near-best guarantee. 

\begin{theorem}\label{thm:optimality}
Let $f:[-1,1]\rightarrow \mathbb{R}$ be a continuous function and $\underline{c}\in\mathbb{R}^{n+1}$.  If $\tfrac{2}{\pi}(n+2)^2\max_{0\leq i\leq n} \! \left|(L[\underline{c}])_i\right| \!< \!1$, then 
\begin{equation}\label{eq:optbound}
\left\|f-\sum_{j=0}^n c_j U_j \right\|_{1} \leq \left(\frac{1}{1-\frac{2}{\pi} (n+2)^2\max_{0\leq i\leq n} \! \left|(L[\underline{c}])_i\right|}\right) \left\|f-\bestLone\right\|_{1},
\end{equation}
where $U_j$ is the degree $j$ Chebyshev polynomial of the second kind.
\end{theorem}
\begin{proof}
Let $p_n\in\mathcal{P}$ and define $s_p(x) = \pm {\rm sign}( f(x) - p_n(x) )$ so that $\|f-p_n\|_{1} = \int_{-1}^1 s_p(x) (f(x)-p_n(x)) dx$. Then,
\begin{equation} 
\|f-\bestLone\|_{1} \geq \int_{-1}^1 s_p(x)(f(x)-\bestLone(x))dx = \|f - p_n \|_{1} +\int_{-1}^1 s_p(x)(p_n(x)-\bestLone(x))dx.
\label{eq:inequality1} 
\end{equation} 
Therefore, we find that $\|f - p_n \|_{1}\leq \|f-\bestLone\|_{1} + \int_{-1}^1 s_p(x)(\bestLone(x)-p_n(x))dx$. Expanding $\bestLone - p_n$ in a Chebyshev series, we find that
\[
\bestLone(x) - p_n(x) = \sum_{i=0}^n a_i U_i(x), \qquad a_i = \frac{2}{\pi}\int_{-1}^1 (\bestLone(x) - p_n(x))U_i(x)\sqrt{1-x^2} dx. 
\]
Since $|U_i(x)| \leq (i+1)$ for $x\in [-1,1]$~\cite[(18.14.4) \& (18.7.4)]{dlmf}, we have 
\[
\left|a_i\right| \leq \frac{2}{\pi} (i+1) \|\bestLone - p_n \|_{1} \leq \frac{4}{\pi} (i+1) \|f - p_n \|_{1},
\]
where the last inequality comes from the fact that  $\|\bestLone - p_n \|_{1} \leq \|\bestLone-f\|_{1}+\|f-p_n\|_{1}\leq 2\|f-p_n\|_{1}$. It follows that 
\begin{equation} 
 \begin{aligned}
\left|\int_{-1}^1 s_p(x)(\bestLone(x) - p_n(x))dx\right| &= \sum_{i=0}^n |{\RefTwo a_i}| \left|\int_{-1}^1 s_p(x)U_i(x)dx\right|\\
& \leq \frac{2}{\pi} (n+2)^2 \|f - p_n \|_{1} \max_{0\leq i\leq n} \left|\int_{-1}^1 s_p(x)U_i(x)dx\right|,
\end{aligned}
\label{eq:inequality2} 
\end{equation} 
where the inequality holds since $\sum_{i=0}^n (i+1) {\RefTwo = (n+1)(n+2)/2\leq (n+2)^2/2}$. By using~\cref{eq:inequality2} to bound the righthand side of~\cref{eq:inequality1}, the result follows by rearranging.
\end{proof}

\Cref{thm:optimality} shows that one can track the quantity $\max_{0\leq i\leq n} \left|(L[\underline{c}^{(k)}])_i\right|$ for $k\geq 0$ to estimate how close the current Newton iterate is to computing $\bestLone$. In practice, we terminate Newton's method as soon as $\max_{0\leq i\leq n} \left|(L[\underline{c}^{(k)}])_i\right|<10^{-14}\|f\|_{1}$.  It can happen that the initial guess in~\cref{sec:LP} already satisfies the stopping criteria in which case no Newton iterations are computed. 

\section*{Acknowledgments}
We thank Laurent Demanet for discussing the implications of the Remez inequality with us. We also thank Vanni Noferini who was present during the initial discussions of this work. We thank Nick Trefethen and Heather Wilber for reading a draft of this manuscript and improving the text. 

\bibliographystyle{abbrv}
\bibliography{bestL1bib}

\appendix 
\section{Corruption away from the endpoints}\label{sec:middleCorruption} 
\Cref{lem:L1L0} shows that polynomials of degree $\leq n$ cannot be too concentrated in a set of measure $< \min(1,1/(4n^2))$, which is a consequence of the fact that $|p'(x)|\leq n^2 \|p\|_{\infty}$ for any $p \in \mathcal{P}_n$. An alternative bound on the derivative of a polynomial is~\cite[Ch.~5]{borwein2012polynomials}
\[
|p'(x)|\leq \frac{n}{\sqrt{1-x^2}}\|p\|_{\infty}, \qquad -1<x<1
\]
for any $p\in\mathcal{P}_n$. This inequality is better when $x$ is away from $\pm 1$, and suggests that polynomials of degree $\leq n$ are less concentrated in the middle of $[-1,1]$ compared to near $\pm 1$.  This turns out to be the case. 
\begin{theorem}
Let $\Omega_s\subseteq [-1,1]$ with Lebesgue measure $s\geq 0$ and suppose that $\zeta = \max\{|x| : x\in\Omega_s\}$ is such that $1-\zeta\geq 1/n$. For $n\geq 1$, we have
\begin{equation}
  \label{eq:Middle}
\int_{\Omega_s} |p(x)| dx\leq \frac{sn^{3/2}}{(1-\zeta^2)^{1/4}}\int_{-1}^1 |p(x)| dx
\end{equation}
for any polynomial $p\in\mathcal{P}_n$.
\label{thm:HoleInTheMiddle}
\end{theorem}
\begin{proof} 
Let $p\in\mathcal{P}_n$ and let $\|p\|_{\Omega_s}$ denote its absolute maximum in $\Omega_s$. By Bernstein's inequality~\cite[Ch.~5]{borwein2012polynomials} we have that $|p'(x)|\leq n\|p\|_{\infty}/\sqrt{1-\zeta^2}$ for $x\in\Omega_s$ and $|p'(x)|\leq n^2\|p\|_{\infty}$ for $x\in[-1,1]$. 
Let $x^*\in[-1,1]$ be such that $|p(x^*)| = \|p\|_{\infty}$. 
Using these two inequalities, {\RefTwo{we observe that there is an interval $\mathcal{I}\subset[-1,1]$ containing $x^*$ of width at least $1/n^2$ for which $p(x)$ is of the same sign as $p(x^*)$. 
The area of the triangle of width $1/n^2$ and height $|p(x^*)|$
is $I_1=|p(x_*)|/(2n^2)$. 
Next use the same argument for $x^{*,\Omega}\in\Omega_s$ such that $|p(x^{*,\Omega})|=\|p\|_{\Omega_s}$, to obtain 
a triangle with area $\|p\|_{\Omega_s}^2\sqrt{1-\zeta^2}/(2\|p\|_{\infty} n)$. 
Note that since $1-\zeta\geq 1/n$, the two triangles can be chosen to not overlap.}}
We can thus write $\int_{-1}^1 |p(x)|dx \geq I_1 + I_2$. 

Since $\int_{\Omega_s} |p(x)|dx \leq s \|p\|_{\Omega_s}$, we find that 
\[
\int_{\Omega_s} |p(x)|dx \leq \frac{s}{\frac{X(p)}{2n^2}+\frac{\sqrt{1-\zeta^2}}{2nX(p)}}\int_{-1}^1 |p(x)|dx, \qquad X(p) = \frac{\|p\|_{L_{\infty}}}{\|p\|_{\Omega_s}}.
\]
The function $g(x) = x/(2n^2) + \sqrt{1-\zeta^2}/(2n x)$ on $x\geq 0$ is minimized at $x_* = \sqrt{n}(1-\zeta^2)^{1/4}$. The bound in~\cref{eq:Middle} holds since $g(x) \geq g(x_*) = (1-\zeta^2)^{1/4}n^{-3/2}$ for any $x\geq 0$. 
\end{proof} 
Arguing as in \cref{cor:L1L0condition}, \cref{thm:HoleInTheMiddle} means that a corrupted polynomial of degree $n\geq 1$ can be exactly recovered by $\bestLone$ when $s <(1-\zeta^2)^{1/4}n^{-3/2}/2$ and $1-\zeta\geq 1/n$.  For sufficiently large $n$, this is a relaxation of the requirements for exact recovery in~\cref{sec:L0L1} when the corruption is away from $\pm 1$ (see~\cref{fig:CorruptedFunction} (c) and the localized error near $x=\pm 1$ in \cref{fig:abs}).  
Other results in~\cref{sec:L0L1smooth} can be relaxed by using \cref{thm:HoleInTheMiddle} under the restriction that the corruption occurs away from $\pm 1$. In particular, one can show that if $\Omega_s = [-s/2,s/2]$ with $s = n^{-3/2}/32$, then 
\[
\int_{[-1,1]\setminus \Omega_s} \left|f(x) - \bestLone(x)\right| dx\leq 4 \|f_0 - p_n^*\|_{1}. 
\]
where $p_n^*$ is the best $L_1$ polynomial approximation of $f_0$ on $[-1,1]$. 

\Cref{thm:HoleInTheMiddle} also encourages us to wildly speculate {\RefTwo{(recalling the derivation of~\cref{cor:L1L0condition})}}
that the error localization of $f - \bestLone$ is usually more concentrated for functions with endpoint singularities, i.e., $|\Omega_n| = \mathcal{O}(n^{-2})$, and less concentrated for functions with singularities away from $\pm 1$, i.e., $|\Omega_n| = \mathcal{O}(n^{-1.5})$ or even $\mathcal{O}(n^{-1})$. 
\end{document}